\documentclass[11pt]{article}

\usepackage{amssymb,latexsym,amsmath}
\usepackage{epsfig}
\usepackage{caption}
 
 \usepackage{amsmath}
\usepackage{amsfonts}
\usepackage{array}
\usepackage{dsfont}
\usepackage{amssymb}
\usepackage{amsthm}

\newtheorem{assumption}{Assumption}

 \newtheorem{remark}{Remark}
  \newtheorem{example}{Example}[section]
\newtheorem{theorem}{Theorem}
\newtheorem{proposition}{Proposition}

\newtheorem{lemma}{Lemma}

\usepackage[colorlinks=true,breaklinks=true,bookmarks=true,urlcolor=blue,
     citecolor=blue,linkcolor=blue,bookmarksopen=false,draft=false]{hyperref}

\def\Pr{\mathop{\rm Pr}}

\def\B{{\mathcal B}}

\def\P{{\mathcal P}}

\newcommand{\dd}{\mathrm{d}}

\newcommand{\adk}[1]{{\color{black} #1}}

\usepackage[utf8]{inputenc}

\begin{document}

\sloppy
\title{Infinite Horizon Average Cost Optimality Criteria for Mean-Field Control\thanks{
E. Bayraktar is partially supported by the National Science Foundation under grant DMS-2106556 and by the Susan M. Smith chair.}
}
\author{ Erhan Bayraktar and Ali Devran Kara
\thanks{
 E. Bayraktar and A.D. Kara are with the Department of Mathematics,
     University of Michigan, Ann Arbor, MI, USA,
     Email:  \{erhan,alikara\}@umich.edu}
     }
\maketitle
\begin{abstract}
We study mean-field control problems in discrete-time under the infinite horizon average cost optimality criteria. We focus on both the finite population and the infinite population setups. We show the existence of a solution to the average cost optimality equation (ACOE) and the existence of optimal stationary Markov policies for finite population problems under (i) a minorization condition that provides geometric ergodicity on the collective state process of the agents, and (ii) under standard Lipschitz continuity assumptions on the stage-wise cost and transition function of the agents when the Lipschitz constant of the transition function satisfies a certain bound. For the infinite population problem, we establish the existence of a solution to the ACOE, and the existence of optimal policies under the continuity assumptions on the cost and the transition functions. Finally, we relate the finite population and infinite population control problems: (i) we prove that the optimal value of the finite population problem converges to the optimal value of the infinite population problem as the number of agents grows to infinity; (ii) we show that the accumulation points of the finite population optimal solution corresponds to an optimal solution for the infinite population problem, and finally (iii), we show that one can use the solution of the infinite population problem for the finite population problem symmetrically across the agents to achieve near optimal performance when the population is sufficiently large. 
\end{abstract}


\section{Introduction}
We focus on optimal control of a team problem where the agents are only correlated through the so-called mean-field term, i.e. the distribution of the states of all agents. We aim to analyze the optimal control problem under the infinite horizon average cost criteria in discrete time.

The dynamics of the model are presented as follows: suppose that $N$ agents (decision makers or controllers) act cooperatively to minimize a cost function, and the agents share a common state and an action space denoted by $\mathds{X}\subset \mathds{R}^l$ and $\mathds{U}\subset{\mathds{R}}^m$ for some $l,m<\infty$. For any time step $t$, and agent $i\in\{1,\dots,N\}$ we have
\begin{align*}
x^i_{t+1}=f(x_t^i,u_t^i,\mu_{\bf x_t},w_t^i,w_t^0)
\end{align*}
for a measurable function $f$, where $\{w_t^i\}$ denotes the i.i.d. idiosyncratic noise process, and $\{w_t^0\}$ denotes the  i.i.d. common  noise process. Furthermore, $\mu_{\bf x}$ denotes the distribution of the agents on the state space $\mathds{X}$ such that for a given joint state ${\bf x}:=\{x^1,\dots,x^N\}\in \mathds{X}^N$, $\mu_{\bf x}:=\frac{1}{N}\sum_{i=1}^N\delta_{x^i}.$

Equivalently, for a given realization of the common noise, $w_t^0$, the next state of the agent $i$ is determined by some stochastic kernel $\mathcal{T}^{w^0_t}(\cdot|x_t^i,u_t^i,\mu_{\bf x_t})$ such that 
\begin{align}\label{kernel_common_noise}
& \Pr\biggl( X^i_{t+1}\in B \, \bigg|\, (X^j,U^j)_{[0,t]}=(x^j,u^j)_{[0,t]}, \forall j=1,\dots,N \biggr)\nonumber \\
& \qquad = \int \int_B \mathcal{T}^{w_t^0}( dx^i_{t+1}|x^i_{t}, u^i_{t},\mu_{\bf x_t})Pr(dw_t^0),  B\in \mathcal{B}(\mathds{X}), t\in \mathds{N},
\end{align}
that is, the conditional probability distribution of the state of some agent $i$, given the state and action history of all agents, depends only on the most recent state and action of agent $i$, and the state distribution of other agents by conditioning on the realizations of the common noise.  

For the remainder of the paper, by an abuse of notation, we will sometimes also denote the dynamics in terms of the vector state and action variables, ${\bf x}=(x^1,\dots,x^N)$, and ${\bf u}=(u^1,\dots,u^N)$, and vector noise variables ${\bf w}=(w^0,w^1,\dots,w^N)$ such that
\begin{align*}
{\bf x_{t+1}}=f({\bf x_t,u_t,w_t}).
\end{align*}

At each time stage $t$, each agent receives a cost determined by a measurable stage-wise cost function $c:\mathds{X}\times\mathds{U}\times \mathcal{P}_N(\mathds{X})\to \mathds{R}$, where $\mathcal{P}_N(\mathds{X})$ is the set of all empirical measures on $\mathds{X}$ constructed using $N$ dimensional state vectors.  That is, at time $t$, agent $i$ observes the cost
\begin{align*}
c(x^i_t,u^i_t,\mu_{\bf x_t}).
\end{align*}

For the initial formulation, every agent is assumed to know the state and action variables of every other agent. We define an admissible policy for an agent $i$, as a sequence of functions $\gamma^i:=\{\gamma^i_t\}_t$, where $\gamma^i_t$ is a $\mathds{U}$-valued (possibly randomized) function which is measurable with respect to the $\sigma$-algebra generated by 
\begin{align*}
I_t=\{{\bf x}_0,\dots,{\bf x}_t,{\bf u}_0,\dots,{\bf u}_{t-1}\}.
\end{align*}

Accordingly, an admissible {\it team} policy, is defined as $\gamma:=\{\gamma^1,\dots,\gamma^N\}$, where $\gamma^i$ is an admissible policy for the agent $i$. In other words, agents share the complete information.

The objective of the agents is to minimize the following infinite horizon average cost function
\begin{align*}
&J^N_\infty({\bf x}_0,\gamma):=\limsup_{T\to\infty}\frac{1}{T}\sum_{t=0}^{T-1}E_\gamma\left[{\bf c}({\bf X}_t,{\bf U}_t)\right]\\
&\text{where } {\bf c}({\bf x}_t,{\bf u}_t):=\frac{1}{N}\sum_{i=1}^Nc(x^i_t,u^i_t,\mu_{{\bf x}_t}).
\end{align*}
The optimal cost is defined by
\begin{align}\label{opt_cost}
J_\infty^{N,*}({\bf x}_0):=\inf_{\gamma\in\Gamma}J^N_\infty({\bf x}_0,\gamma)
\end{align}
where $\Gamma$ denotes the set of all admissible team policies.

We note that the value function that is achieved when agents share full information, as presented here, will be taken to be our reference point for simpler information structures.

\subsection{Literature Review}
Multi-agent systems attract attention due to their application flexibility to real world problems. Many problems that can be modeled as a multi-agent system may be intractable to model as a single agent system. Although they offer a practically rich and relevant model, the mathematical analysis is much more challenging compared to the single agent systems in general.  
The dynamics we consider in this paper represents a class of multi-agent systems which we will refer to as mean-field type multi-agent problems. For the mean-field type multi-agent systems, the population
is homogeneous and weakly interacting such that the dynamics of each
agent depends on the other agents only through the state distribution
of all the other agents.  The motivation of this formulation comes from systems where the effect of a single agent on the whole system is minimal (which is usually the case for large population systems). The mathematical tractability of the mean-field systems mainly relies on the observation that this weak coupling between the agents disappears through {\it propagation of chaos}.

If agents are non-cooperative, the mean-field systems in general are studied under mean-filed game theory. If the agents are cooperative, these problems are referred to as mean-field control. Mean-field game theory has been introduced independently by \cite{huang2006large,lasry2007mean}. There has been a great progress and a significant number of publications since then. For some of the work that has been done since, we cite \cite{gomes2014mean,carmona2013probabilistic,bensoussan2013mean,tembine2013risk,huang2007large,anahtarci2022q,elie2020convergence,fu2019actor,guo2019learning, perrin2020fictitious,subramanian2019reinforcement,saldi2018markov, saldi2019approximate} and references therein for mean-field game theory studies both in continuous and discrete time.

For the cooperative case, i.e. for mean-field control, similarly, there have been a number of studies dealing with structural as well as computational and learning-related methods. See \cite{bayraktar2018randomized,djete2022mckean, djete2022mckean1, lauriere2014dynamic,pham2017dynamic,carmona2021convergence,germain2022numerical,bayraktar2021mean,bayraktar2021solvability,bayraktar2023finite,sanjari2024optimality, sanjari2022optimality, sanjari2021optimal,lacker2017limit,fornasier2019mean,motte2022quantitative,motte2022mean,carmona2019model,gu2021mean,gu2019dynamic,angiuli2022unified} and the references therein for the study of the dynamic programming principle, learning methods, and justification of the exchangeability of agents for large (possibly infinite) agent team settings.

We note that most of the papers we have cited so far deal with either the finite horizon or the infinite horizon discounted cost criteria. The standard tools used to study these problems mostly involve the dynamic programming principle or the contraction property of the Bellman operator via the discount factor. However, the analysis of the infinite horizon average cost optimality criteria requires significantly different tools than the ones used in the study of finite horizon or discounted cost criteria. For the average cost criteria, one usually needs certain stability (or ergodicity) properties of the controlled process, to ensure the existence of a solution. 

There are several works studying the mean-field game problems, i.e. the non-cooperative case, in continuous time under infinite horizon average cost criteria, see \cite{feleqi2013derivation,cardaliaguet2012long,arapostathis2017solutions,bardi2014linear,cirant2016stationary,dragoni2018ergodic,gomes2014existence}. We note that in the game problems, as the agents are self interested, it is relatively more direct to decouple the agents with respect to their cost. However, the same is not true for the mean-field control problems as the goal is the minimization of the social optima and thus the agents are coupled via their common cost functions as well. Furthermore, the state measure flow can be viewed as an exogenous flow for the game problems, and thus the stability and the ergodicity of the agent state processes can be analyzed separately. For the control problems, the stability of the state measure flow itself needs to be analyzed. 

\adk{
For continuous time mean-field control problems, we refer the reader to \cite{albeverio2022mean,bao2023ergodic} for the ergodic control of Mckean-Vlasov dynamics. \cite{albeverio2022mean} studies a class of team control problems where the dynamics of the agents does not depend on the mean-field term, however, the stage-wise cost function of the problem is dependent on the distribution of the states of the agents. In particular, the agent dynamics are coupled only through the control. The authors establish the precise relationship between  the finite population team problem, and the infinite population control problem, where the agents are decoupled and can use feedback policies with their local states. \cite{bao2023ergodic} considers a general model where the mean-field term can affect the dynamics via drift and diffusion. The authors do not consider the finite population problem and work directly with the limiting McKean-Vlasov dynamics, and show the existence of the unique viscosity solution for the ergodic control problem under a class of deterministic control functions which are assumed to be Lipschitz continuous in state and the mean-field term.}
 
In discrete-time, the number of studies that focus on the average-cost optimality criteria for mean-field problems is limited compared to the continuous-time studies. Furthermore, most of the existing work focuses on the mean-field game problems. \cite{saldi2020discrete,anahtarci2020value,biswas2015mean,wikecek2015stationary,wikecek2020discrete} are among the papers which study ergodic mean-field game problems in discrete time. \cite{wikecek2015stationary} considers mean-field game problems with discrete spaces. \cite{biswas2015mean,wikecek2020discrete} establish mean-field equilibria using the ergodicity properties of the agent state process under a mixing type condition on the agent dynamics. \cite{biswas2015mean}  also assumes that the dynamics do not depend on the mean-field of the agent states. \cite{saldi2020discrete,anahtarci2020value} provide value iteration algorithms for mean-field game problems where the convergence of these iterations are shown to the mean-field equilibria under certain minorization and mixing type conditions on the agent dynamics. We note that the papers that focus on the game setup aim to establish the mean field equilibria. Therefore, they are able to use the ergodicity properties of the dynamics by using stationary policies for the agents (independent of the mean-field term) under minorization conditions on the transition kernel. However, for mean-field control problems, the focus is mainly on establishing the existence of the global optimal performance for the team problem. In particular, the resulting measure valued control problem requires the use of policies that depend on the mean-field term (empirical distribution of the states). Hence, the minorization condition by itself may not be sufficient to establish the convergence to the stationary distributions of the collective agent states. Hence, in this paper, we make use of weak continuity assumptions on the agent dynamics without needing the ergodicity or minorization assumptions on the state process to establish the existence of optimal stationary and Markov policies.

Finally, to our knowledge, the only paper that deals with the discrete-time mean-field control problems under average cost (reward) criteria is \cite{bauerle2021mean}. In \cite{bauerle2021mean}, existence of optimal stationary Markov policies has been established through the average cost optimality inequality. It is assumed that the difference between the optimal value of any two initial points is uniformly bounded over all initial points and discount factors. This assumption is usually not easy to verify using conditions on  the primitive system components. Furthermore, in \cite{bauerle2021mean}, two special cases are considered, namely (i) when the cost function only depends on the mean-field term and the dynamics are independent of the mean-field term, and there is no common noise (ii)  when the cost functions only depend on the mean-field term (transitions can depend on the mean field term) and the admissible policies are Lipschitz continuous in the mean-field term. It is shown that the dynamic optimization problem reduces to a static optimization problem under either of these special settings. 

{\bf Contributions.}  Our contributors are itemized as follows:
\begin{itemize}
\item In Section \ref{meas_valued_sec}, we focus on the finite population control problem.
\begin{itemize}
\item In Section \ref{fin_rel_op}, we establish the existence of a solution to the average cost optimality equation (ACOE) and the existence of optimal stationary Markov policies under a minorization condition on the agent dynamics (Assumption \ref{main_assmp}). We also provide a relative value iteration algorithm and its convergence with a convergence rate. However, the convergence speed diminishes as the number of agents grows and the convergence may even fail for the infinite population problems.
\item In Section \ref{weak_finite}, we show the existence of a solution to the ACOE and the existence of optimal policies without the mixing (or minorization) condition. Instead, we work with the standard continuity conditions used to study mean-field control problems. Namely, we assume that the stage-wise cost function and the agent transition function are Lipschitz continuous, and the state and action spaces are compact (see Assumption \ref{lip_reg}). We note that avoiding the use of strong ergodicity assumptions (e.g. Assumption \ref{main_assmp}) on the state process is a significant relaxation.
\end{itemize}
\item In Section \ref{inf_problem}, we focus on the infinite population control problem, which can be modeled as a measure valued single agent control problem. Even though the problem can be seen as a Markov decision process (MDP), the standard tools used to study the average cost optimality criteria fail for the measure valued mean-field MDP problem (see Remark \ref{caution_remark} for more detail). Hence, we utilize different tools:
\begin{itemize}
\item In Section \ref{weak_inf}, we show that Assumption \ref{lip_reg}, that is the continuity of the transitions and the stage-wise cost function, is also sufficient for the infinite population problem to establish the existence of ACOE and the existence of optimal policies. 
\end{itemize} 
\adk{
\item Finally, in Section \ref{N_to_inf}, we focus on the relationship between the finite population and the infinite population control problems.
\begin{itemize}
\item In Section \ref{val_conv_inf} we show that the optimal average cost value of the finite population problem converges to the optimal average cost value of the infinite population problem as the number of agents grows to infinity.
\item In Section \ref{accum}, we show that the limit points of the convergent subsequences of the measure flow for the finite population problem (as $N$ grows) correspond to an optimal measure flow for the infinite population problem.
\item In Section \ref{sym_pol_sec}, we first provide an example that shows that the optimal policies for the finite population setting may have to be personalized and asymmetric. We then establish the near optimality of the symmetric policies designed for the infinite population problem, when they are used for the finite population problem under the discounted cost criteria, for sufficiently large $N$. We finally show that if the discount factor $\beta$ is sufficiently close to $1$, then this symmetric policy will achieve near optimal performance under the ergodic cost criteria as well.
\end{itemize}
}
\end{itemize}

\subsection{Metrics for Probability Measures} For the analysis of the technical results, we will make use of different distance notions for probability measures; total variation distance and the Wasserstein metric.

For some Polish,  a separable completely metrizable topological space $\mathds{X}$ and  for probability measures $\mu,\nu\in\mathcal{P}(\mathds{X})$, the \emph{total variation} metric is given by
  \begin{align*}
    \|\mu-\nu\|_{TV}&=2\sup_{B\in\mathcal{B}(\mathds{X})}|\mu(B)-\nu(B)|=\sup_{f:\|f\|_\infty \leq 1}\left|\int f(x)\mu(\dd x)-\int f(x)\nu(\dd x)\right|,
  \end{align*}

  \noindent where the supremum is taken over all measurable real valued $f$ such that $\|f\|_\infty=\sup_{x\in\mathds{X}}|f(x)|\leq 1$.

The other distance notion we will use in the paper is the Kantorovich-Rubinstein distance. For two probability measures  $\mu,\nu\in\mathcal{P}(\mathds{X})$, the Kantorovich-Rubinstein distance between them can be written as
\begin{align*}
W_1(\mu,\nu)=\sup_{\|f\|_{Lip}\leq 1} \left| \int f(x)\mu(dx)-\int f(x)\nu(dx)\right|
\end{align*}
where $\|f\|_{Lip}$ denotes the minimal Lipschitz constant of $f$. We note that this distance notion is equivalent to first order Wasserstein metric via the Kantorovich-Rubinstein duality. For the rest of the paper, we will sometimes refer to this metric as the Wasserstein distance.

\section{Ergodic Control of Finite Populations}\label{meas_valued_sec}
For the rest of the paper, we will analyze the problem by considering the controlled process to be the state distribution of the agents, rather than the state vector of the agents.

In this section, we will define an MDP for the distribution of the agents, where the control actions are the joint distribution of the state and action vectors of the agents. 

We let the state space to be  $\mathcal{Z}=\P_N(\mathds{X})$ which is the set of all empirical measures on $\mathds{X}$ that can be constructed using the state vectors of  $N$-agents. In other words, for a given state vector ${\bf x}=\{x^1,\dots,x^N\}$, we consider $\mu_{\bf x}\in \P_N(\mathds{X})$ to be the new state variable of the team control problem.

The admissible set of actions for some state $\mu\in\mathcal{Z}$, is denoted by $U(\mu)$, where
\begin{align}\label{add_act}
U(\mu)=\{\Theta\in \P_N(\mathds{U}\times\mathds{X})|\Theta(\mathds{U},\cdot)=\mu(\cdot)\},
\end{align}
that is, the set of actions for a state $\mu$, is the set of all joint empirical measures on $\mathds{X}\times\mathds{U}$ whose marginal on $\mathds{X}$ coincides with $\mu$.

We equip the state space $\mathcal{Z}$, and the action sets $U(\mu)$, with the first order Wasserstein distance $W_1$.

In order to  define the transition model for this centralized MDP, we note that the empirical distributions of the agents' states of the original team problem induces a controlled Markov chain. In particular, for some set $B\in \B(\mathcal{Z})$, we can write
\begin{align*}
&Pr(\mu_{t+1}\in B|\mu_t,\dots\mu_0,\Theta_t,\dots,\Theta_0)\\
&\qquad=\int_{{\bf x_t,u_t}\in \mathds{X}^N\times\mathds{U}^N}Pr(\mu_{t+1}\in B|{\bf x_t,u_t})Pr({\bf dx_t,du_t}|\mu_t,\dots\mu_0,\Theta_t,\dots,\Theta_0).
\end{align*}
For any $\mu_{\bf x_t,u_t}=\Theta_t$, and $\mu_{\bf x_t}=\mu_t$, the inside term can be written as
\begin{align*}
Pr(\mu_{t+1}\in B|{\bf x_t,u_t})=\int\mathds{1}_{\{\mu_{f({\bf x_t,u_t,\mu_{x_t},w_t})}\in B\}}P(d{\bf w_t})
\end{align*}
where $P(\cdot)$ is the probability measure governing the idiosyncratic and the common noise processes, and ${\bf w_t}$ is the noise vector with length $N+1$. 

If any two pairs $({\bf x_t,u_t}),({\bf x'_t,u'_t})$ have the same empirical distribution $\Theta_t$, then they can be viewed as reordered versions of each other. Furthermore, since the dynamics are identical for every agent, i.e. since the agent are exchangeable, for some ${\bf w_t}$, and ${\bf x_{t+1}}=f({\bf x_t,u_t,\mu_{\bf x_t}w_t})$, where $\mu_{\bf x_{t+1}}=\mu_{t+1}$, by reordering ${\bf w_t}$, one can construct some ${\bf w'_t}$ such that ${\bf x'_{t+1}}=f({\bf x'_t,u'_t,\mu_{\bf x_t},w'_t})$, where ${\bf x'_{t+1}}$ is just a reordered version of ${\bf x_{t+1}}$, and in particular $\mu_{\bf x_{t+1}}=\mu_{\bf x'_{t+1}}$.

Since the idiosyncratic noises are identically distributed for every agent, as a result of the above discussion, for any two pairs $({\bf x_t,u_t}),({\bf x'_t,u'_t})$ with the same empirical distribution $\Theta_t$, 
\begin{align*}
Pr(\mu_{t+1}\in B|{\bf x_t,u_t})=Pr(\mu_{t+1}\in B|{\bf x'_t,u'_t}).
\end{align*}

Therefore, the empirical distributions of the agents' states $\mu_t$, and of the joint state and actions $\Theta_t$ define a controlled Markov chain such that
\begin{align}\label{trans_measure}
&Pr(\mu_{t+1}\in B|\mu_t,\dots\mu_0,\Theta_t,\dots,\Theta_0)=Pr(\mu_{t+1}\in B|\mu_t,\Theta_t)\nonumber\\
&:=\eta(B|\mu_t,\Theta_t)\\
&=Pr(\mu_{t+1}\in B|{\bf x_t,u_t}), \text{ for any } ({\bf x_t,u_t}): \mu_{({\bf x_t,u_t})}=\Theta_t\nonumber
\end{align}
where $\eta(\cdot|\mu,\Theta)\in \P(\P_N(\mathds{X}))$ is the transition kernel of the centralized measure valued MDP, which is induced by the dynamics of the team problem.

We define the stage-wise cost function $k(\mu,\Theta)$ by
\begin{align}\label{cost_measure}
k(\mu,\Theta):=\int c(x,u,\mu)\Theta(du,dx)=\frac{1}{N}\sum_{i=1}^Nc(x^i,u^i,\mu).
\end{align}



Thus, we have an MDP with state space $\mathcal{Z}$, action space $\cup_{\mu\in\mathcal{Z}}U(\mu)$, transition kernel $\eta$ and the stage-wise cost function $k$.

We define the set of admissible policies for this measured valued MDP as a sequence of functions $g=\{g_0,g_1,g_2,\dots\}$ such that at every time $t$, $g_t$ is measurable with respect to the $\sigma$-algebra generated by the information variables
\begin{align*}
I_t=\{\mu_0,\dots,\mu_t,\Theta_0,\dots,\Theta_{t-1}\}.
\end{align*}
We denote the set of all admissible control policies by $G$ for the measure valued MDP.

In particular, we define the infinite horizon average expected cost function under a policy $g$ by
\begin{align*}
K^N_\infty(\mu_0,g)=\limsup_{T\to\infty}\frac{1}{T}\sum_{t=0}^{T-1}E_{\mu_0}^\eta\left[ k(\mu_t,\Theta_t)\right].
\end{align*}

We also define the optimal cost by
\begin{align}\label{opt_cost2}
K_\infty^{N,*}(\mu_0)=\inf_{g\in G}K^N_\infty(\mu_0,g).
\end{align}
\begin{remark}
We note that the measure valued control problem considered in this section, and the original team control problem are equivalent.  Any admissible policy for the original team problem can be realized as an admissible policy for the measure valued MDP problem, vice versa, and they achieve the same accumulated expected cost. In particular, the optimal value of two problems coincide. For further discussion see \cite{bayraktar2023finite,bauerle2021mean}.
\end{remark}

\subsection{Average Cost Optimality Equation for the Measure Valued MDP for Finite Populations}

The average cost optimality equation (ACOE) is a common tool to analyze the optimality of the average cost criteria for MDPs. In particular for the model described in the last section, the ACOE is given by
\begin{align}\label{acoe}
j^*+h(\mu)=\inf_{\Theta\in U(\mu)}\bigg(k(\mu,\Theta)+\int_{\P_N(\mathds{X})} h(\mu_1)\eta(d\mu_1|\mu,\Theta)\bigg),
\end{align}
for some  $h\in \B(\P_N(\mathds{X}))$ where $ B(\P_N(\mathds{X}))$ denotes the set of bounded and measurable functions on $\P_N(\mathds{X})$ and some constant $j^*$ where $U(\mu)$ denotes the set of admissible actions for the empirical distribution $\mu$. If  a solution to the ACOE  exists, that is if there exists some $j^*$ and $h$ that satisfy the ACOE, then the optimal value function is given by the constant $j^*$.
We present the following result for completeness and for future references:
\begin{theorem}\label{ver_thm}\cite{her89}
Suppose there exists  $h \in \B(\P_N(\mathds{X}))$ and some constant $j^*$ such that 
\begin{align*}
j^*+h(\mu)=\inf_{\Theta\in U(\mu)}\bigg(k(\mu,\Theta)+\int_{\P_N(\mathds{X})} h(\mu_1)\eta(d\mu_1|\mu,\Theta)\bigg),
\end{align*}
we then have
\begin{itemize}

\item[(i)] for any initial point $\mu\in\P_N(\mathds{X})$, the constant $j^*$ is the optimal infinite horizon average cost, that is
\begin{align*}
j^*=K^{N,*}_\infty(\mu)=\inf_{g\in G}K^N_\infty(\mu,g)
\end{align*}
for every $\mu\in \P_N(\mathds{X})
$.
\item[(ii)] If there exists a policy $g^*\in G$ achieving the minimum on the right hand side for every $\mu$, then this stationary policy is an optimal policy for the average infinite horizon cost problem; that is, if $g^*$ satisfies
\begin{align*}
j^*+h(\mu)=k(\mu,g^*(\mu))+\int_{\P_N(\mathds{X})} h(\mu_1)\eta(d\mu_1|\mu,g^*(\mu))
\end{align*}
then $K^{N}_\infty(\mu,g^*)=K_\infty^{N,*}(\mu)$.
\end{itemize}
\end{theorem}

Thus, if one can guarantee the existence of a solution to the ACOE (\ref{acoe}), for the $N$-population measure valued MDP, existence of an optimal stationary policy for the distribution of the population can also be shown under appropriate measurable selection conditions. We will establish the existence of a solution to the ACOE under two different sets of assumptions. For the first case, we will assume a minorization condition on the dynamics of the agents, which will in turn give us a mixing type result for the state vector of all agents. For the second case, we will assume that the cost and transition functions of the agents are Lipschitz continuous without assuming the minorization condition.
\begin{remark}
We note that the solution to the ACOE may not be unique; it is clear that $j^*$ is unique to the problem as it is equal to the optimal value; however, there may be several $h$ that satisfy the ACOE. In fact any shifted version of $h$, e.g. $h'=h + M$ for some $M<\infty$ will also satisfy the ACOE. 
\end{remark}

\subsection{Existence of a Solution to ACOE under Minorization and Geometric Ergodicity}\label{fin_rel_op}
We introduce a relative value iteration approach that converges to the ACOE in (\ref{acoe}). Define the operator $T: \B(\P_N(\mathds{X}))\to  \B(\P_N(\mathds{X}))$  such that for $v \in  \B(\P_N(\mathds{X}))$
\begin{align}\label{operator}
Tv(\mu):=\inf_{\Theta\in U(\mu)}\bigg(k(\mu,\Theta)+\int_{\P_N(\mathds{X})} v(\mu_1)\eta(d\mu_1|\mu,\Theta)\bigg),
\end{align}
Using this, we now define a relative operator,   $T_0: \B(\P_N(\mathds{X}))\to  \B(\P_N(\mathds{X}))$, such that for some fixed $\mu_0\in\P_N(\mathds{X})$
\begin{align}\label{rel_op}
T_0v(\mu):=Tv(\mu)-Tv(\mu_0).
\end{align}


In what follows, we focus on the contraction property of the relative operator $T_0$.

Recall the stochastic kernel (see (\ref{kernel_common_noise})), $\mathcal{T}^{w^0}(\cdot|x_t^i,u_t^i,\mu_{\bf x_t}) \in \P(\mathds{X})$ for the agent dynamics.
\begin{assumption}\label{main_assmp}
There exists a non-trivial measure $\pi(\cdot)\in\B(\mathds{X})$ with $\pi(\mathds{X})>0$, and there exists measurable set $B \in \B(\mathds{W})$ with positive probability, $P(B)>0$, such that for any $x_t^i,u_t^i,\mu_{\bf x_t}$ we have 
\begin{align*}
\mathcal{T}^{w^0}(\cdot|x_t^i,u_t^i,\mu_{\bf x_t})\geq\pi(\cdot)
\end{align*}
for all $w^0\in B$.
\end{assumption}

\begin{theorem}\label{erg_acoe}
Under Assumption \ref{main_assmp}, we have that
\begin{align*}
\|T_0^k v_0-h\|_\infty\to 0
\end{align*}
for any $v_0\in\B(\P_N(\mathds{X}))$ and for some $h\in\B(\P_N(\mathds{X}))$, where $T_0^k$ denotes the operator $T_0$ defined by (\ref{rel_op}) applied $k$-times. Furthermore, $h$ satisfies the following ACOE
\begin{align*}
j^*+h(\mu)=\inf_{\Theta\in U(\mu)}\bigg(k(\mu,\Theta)+\int_{\P_N(\mathds{X})} h(\mu_1)\eta(d\mu_1|\mu,\Theta)\bigg).
\end{align*}
\end{theorem}

\begin{proof}
We define the {\it span} semi-norm for a function $f$ such that
\begin{align*}
\|f\|_{sp}:=\sup_{x}f(x)-\inf_x f(x).
\end{align*}
We will show that the operator $T_0$ is a contraction under the span semi-norm.

Let $f,g$ be measurable functions and $\mu,\mu'\in\P_N(\mathds{X})$. Then by the definition of $T_0$, we can write:
\begin{align*}
&|\left(T_0f(\mu)-T_0g(\mu)\right)-\left(T_0f(\mu')-T_0g(\mu')\right)|\\
&=|(Tf(\mu)-Tf(\mu_0)) - (Tg(\mu)-Tg(\mu_0))-(Tf(\mu')-Tf(\mu_0)) + (Tg(\mu')-Tg(\mu_0)) |\\
&\phantom{}=|(Tf(\mu)-Tg(\mu)) - (Tf(\mu')-Tg(\mu'))|\\
&=\Bigg|\inf_{\Theta\in U(\mu)}\bigg(k(\mu,\Theta)+\int_{\P_N(\mathds{X})} f(\mu_1)\eta(d\mu_1|\mu,\Theta)\bigg)\\
&\phantom{xxxx}-\inf_{\Theta\in U(\mu)}\bigg(k(\mu,\Theta)+\int_{\P_N(\mathds{X})} g(\mu_1)\eta(d\mu_1|\mu,\Theta)\bigg)\\
&\qquad\qquad-\Bigg(\inf_{\Theta\in U(\mu')}\bigg(k(\mu',\Theta)+\int_{\P_N(\mathds{X})} f(\mu_1)\eta(d\mu_1|\mu',\Theta)\bigg)\\
&\phantom{xxxxxxxxxxxxxxx} - \inf_{\Theta\in U(\mu')}\bigg(k(\mu',\Theta)+\int_{\P_N(\mathds{X})} g(\mu_1)\eta(d\mu_1|\mu',\Theta)\bigg)\Bigg)\Bigg|.
\end{align*}
Assume now without loss of generality that the first difference is greater than the second difference, and denote the $\epsilon$-near achieving (since the existence of the minimizers is not guaranteed) action selections by $\Theta^f, \Theta^g, \Theta^f_0,\Theta_0^g$ respectively in the order of the terms in the equation. We can then find an upper-bound as:
\begin{align}\label{contrac_bound}
&|(Tf(\mu)-Tg(\mu)) - (Tf(\mu')-Tg(\mu'))|\nonumber\\
&\leq \bigg(k(\mu,\Theta_g)+\int_{\P_N(\mathds{X})} f(\mu_1)\eta(d\mu_1|\mu,\Theta_g)-k(\mu,\Theta_g)-\int_{\P_N(\mathds{X})} g(\mu_1)\eta(d\mu_1|\mu,\Theta_g)\bigg)\nonumber\\
&-\bigg(k(\mu',\Theta'_f)+\int_{\P_N(\mathds{X})} f(\mu_1)\eta(d\mu_1|\mu',\Theta'_f)-k(\mu',\Theta'_f)-\int_{\P_N(\mathds{X})} g(\mu_1)\eta(d\mu_1|\mu',\Theta'_f)\bigg) + \epsilon\nonumber\\
&\leq \int_{\P_N(\mathds{X})} f(\mu_1)-g(\mu_1)\eta(d\mu_1|\mu,\Theta_g)-\int_{\P_N(\mathds{X})} f(\mu_1)-g(\mu_1)\eta(d\mu_1|\mu',\Theta'_f) +\epsilon.
\end{align}
For any  $\mu,\mu'\in\P_N(\mathds{X})$ and any $\Theta\in U(\mu),\Theta' \in U(\mu')$ there exists state and action vectors ${\bf x,x',u,u'}$ such that $\mu_{\bf x}=\mu,\mu_{\bf x'}=\mu'$ and $\mu_{\bf (x,u)}=\Theta, \mu_{\bf (x',u')}=\Theta'$. Furthermore, following the definition of $\eta$ (see (\ref{trans_measure})), we have that for any $A\in \B(\P_N(\mathds{X}))$
\begin{align}\label{eta_vector}
&\eta(A|\mu,\Theta)=Pr(\mu_1\in A|{\bf x,u})=Pr(\{{\bf x_1\in\mathds{X}}:\mu_{\bf x_1}\in A\}|{\bf x,u})\nonumber\\
&\eta(A|\mu',\Theta')=Pr(\mu_1\in A|{\bf x',u'})=Pr(\{{\bf x_1\in\mathds{X}}:\mu_{\bf x_1}\in A\}|{\bf x',u'}).
\end{align}
Note that $\mathds{X}^N$ and $\mathds{U}^N$ are equipped with the product topology, and the vector state process ${\bf x}_t\in\mathds{X}^N$ is also a controlled Markov chain, where the control variables are ${\bf u}_t\in\mathds{U}^N$ are simply the action vector variables. We write the transition kernel of the vector state process as $P({\bf x}_{t+1}\in \cdot|{\bf x_t,u_t})$.

Consider $\pi(\cdot)\in\P(\mathds{X})$ from Assumption \ref{main_assmp}, we define the product measure $\pi^N(\cdot)\in\P(\mathds{X}^N)$ using $\pi(\cdot)$, that is, for any rectangular set $E=E_1\times\dots\times E_N$ where $E_1,\dots,E_N\subset \mathds{X}$, we have that 
\begin{align*}
\pi^N(E)=\pi(E_1)\times\dots\times\pi(E_N).
\end{align*} 
Furthermore, using the conditional independence of the agent states, $x^j$ for $j=1,\dots,N$, given ${\bf x,u}$ (or given $x^j,u^j,\mu_{\bf x}$), and given the common noise realization we can also write that 
\begin{align*}
P({\bf x_1}\in E|{\bf x,u})=\int\prod_{j=1}^N \mathcal{T}^{w^0}(x^j_1 \in E_j|x^j,u^j,\mu_{\bf x})P(dw^0)
\end{align*}

Consider a general measurable set $A\in\B(\mathds{X}^N)$, which can be written as a countable union of disjoint rectangular sets, say $\{E^i\}_i$, such that $\cup_iE^i=A$. We can then write 
\begin{align*}
P({\bf x_1}\in A|{\bf x,u})&=\sum_{i}P(E^i|{\bf x,u})=\sum_i \int \prod_{j=1}^N \mathcal{T}^{w^0}(E^i_j|x^j,u^j,\mu_{\bf x})P(dw^0)\\
&\geq \sum_i \prod_{j=1}^N \pi(E^i_j)P(B)=P(B)\sum_i \pi^N(E^i)=P(B)\pi^N(A).
\end{align*}
Note that by construction, $\pi^N(\cdot)$ is a non-trivial measure, and furthermore, $\pi^N(\mathds{X}^N)=\pi(\mathds{X})^N<1$. We denote by $\hat{P}(\cdot|{\bf x,u}):=P(\cdot|{\bf x,u})- P(B)\pi^N(\cdot)$. Going back to (\ref{contrac_bound}) and denoting $h(\mu):=f(\mu)-g(\mu)$, we write
\begin{align*}
&\left(\int_{\P_N(\mathds{X})} f(\mu_1)-g(\mu_1)\eta(d\mu_1|\mu,\Theta_g)-\int_{\P_N(\mathds{X})} f(\mu_1)-g(\mu_1)\eta(d\mu_1|\mu',\Theta'_f)\right) + \epsilon\\
&=\int h(\mu_{\bf x_1}) P(d{\bf x_1}|{\bf x,u_g}) - \int h(\mu_{\bf x_1})P(d\bf{x_1}|{\bf x',u'_f}) + \epsilon\\
& =\int h(\mu_{\bf x_1}) \hat{P}(d{\bf x_1}|{\bf x,u_g}) - \int h(\mu_{\bf x_1})\hat{P}(d\bf{x_1}|{\bf x',u'_f}) + \epsilon\\
&\leq \sup_\mu h(\mu)\hat{P}(\mathds{X}^N|{\bf x,u_g}) - \inf_\mu h(\mu)\hat{P}(\mathds{X}^N|{\bf x',u'_f}) + \epsilon\\
&\leq  \|f-g\|_{sp}\left( 1-P(B)\pi^N(\mathds{X}^N)\right) +\epsilon.
\end{align*}
Since $\mu,\mu'$ and $\epsilon$ are arbitrary, taking $\epsilon\to0$ shows that $T_0$ is a contraction under the span norm. \adk{ Hence, using an extended version of  Banach Fixed Point Theorem under the span semi-norm, one can show that there exists a $h$ such that $\|T_0 h - h\|_{sp}=0$. Note that the fixed point theorem is used under the span semi-norm which is a pseudo metric. Hence, the fixed point is only unique under the quotient space that is defined by the equivalent classes of functions that are invariant under constant shifts. However, for the original space of functions, there might be several fixed points. In other words, any function $v^*$ such that $T_0 h (\mu)- h(\mu)=C$  for some constant $C<\infty$ for all $\mu$, is a fixed point of the operator $T_0$ under the span semi-norm.} We can also write this as
\begin{align*}
h(\mu)=&\inf_{\Theta\in U(\mu)}\bigg(k(\mu,\Theta)+\int_{\P_N(\mathds{X})} h(\mu_1)\eta(d\mu_1|\mu,\Theta)\bigg)\\
&-\inf_{\Theta\in U(\mu_0)}\bigg(k(\mu_0,\Theta)+\int_{\P_N(\mathds{X})} h(\mu_1)\eta(d\mu_1|\mu_0,\Theta)\bigg) + C.
\end{align*}
Note that $T_0$ is constructed so that $T_0^k v_0 (\mu_0)=0$ for all $k$. Hence,  the fixed point also satisfies $h(\mu_0)=0$. For the above equation, we then must have that $C=0$. Therefore, among the span-norm equivalent set of functions that $T_0^k$ converges to under the span semi-norm, for the one with $h(\mu_0)=0$, we have convergence under the uniform norm as well. That is $\|T_0^k v_0 - h\|_\infty \to 0$, for any $v_0\in \B(\P_n(\mathds{X}))$, such that $h(\mu_0)=0$ and it satisfies the following
\begin{align*}
h(\mu)=&\inf_{\Theta\in U(\mu)}\bigg(k(\mu,\Theta)+\int_{\P_N(\mathds{X})} h(\mu_1)\eta(d\mu_1|\mu,\Theta)\bigg)\\
&-\inf_{\Theta\in U(\mu_0)}\bigg(k(\mu_0,\Theta)+\int_{\P_N(\mathds{X})} h(\mu_1)\eta(d\mu_1|\mu_0,\Theta)\bigg)
\end{align*}
which is in the form of ACOE with
\begin{align*}
j^* = \inf_{\Theta\in U(\mu_0)}\bigg(k(\mu_0,\Theta)+\int_{\P_N(\mathds{X})} h(\mu_1)\eta(d\mu_1|\mu_0,\Theta)\bigg).
\end{align*}

\end{proof}

\begin{remark}\label{inf_fails}
Assumption \ref{main_assmp} requires that under a set of positive probability on the common noise, an agent can go to a subset of the state space $\mathds{X}$, with positive probability decided by the measure $\pi$ no matter what the current state, action, and the empirical measure are. Note that $\pi$ does not have to be a probability measure. However, if $\pi(\mathds{X})<1$, then the contraction constant ${1-P(B)\pi(\mathds{X})^N}$ we use in the proof goes to $1$ as $N$ grows. This suggests that for large (or infinite) populations, Theorem \ref{erg_acoe} no longer has a useful conclusion. One way to avoid this problem is to require $\pi$ to be a probability measure. In other words, if there exists a measurable set $B$ with positive probability, $P(B)>0$, such that for any $x_t^i,u_t^i,\mu_{\bf x_t}$ we have 
\begin{align*}
\mathcal{T}^{w^0}(\cdot|x_t^i,u_t^i,\mu_{\bf x_t})=\pi(\cdot)
\end{align*}
for all $w^0\in B$
where $\pi$ is a probability measure, then Theorem \ref{erg_acoe} will still provide a contraction constant, uniform over the population $N$. This condition can be thought of as resetting the distribution of every agent to $\pi$ under certain common noise realizations. This is clearly a strong assumption, and maybe an artificial one. Thus, in the following we will use a different set of assumptions for the analysis of ACOE which will be valid for large (and infinite) population problems as well.
\end{remark}

\begin{remark}
Assumption \ref{main_assmp}, we use in this section, is also a sufficient condition for geometric ergodicity of the vector state process of the team under stationary policies. Consider two team state processes under the same policy $\gamma$ where one starts from some $\hat{\pi}_0$ and the other starts from $\pi_0$. Denote by $\pi_t^\gamma\in \P(\mathds{X}^N)$ and $\hat{\pi}_t^\gamma\in \P(\mathds{X}^N)$ the marginal distributions of the state vectors at time $t$. One can show that under Assumption \ref{main_assmp}
\begin{align}\label{cauchy_tv}
\|\pi_{t+1}^\gamma - \hat{\pi}_{t+1}^\gamma\|_{TV}\leq \alpha\|\pi_{t+1}^\gamma - \hat{\pi}_{t+1}^\gamma\|_{TV}
\end{align}
where $\alpha=(1-P(B)\pi(\mathds{X})^N)<1$. In particular, one can use this relation to show that $\pi_t^\gamma$ is a Cauchy sequence, and thus converges to some $\pi^\gamma$ under the total variation metric, since $\P(\mathds{X}^N)$ is complete under total variation. The limit point $\pi^\gamma$ is the invariant measure of the vector state process under the stationary team policy $\gamma$. Furthermore, using (\ref{cauchy_tv}), one can also show that for any starting distribution of agents
\begin{align*}
\|\pi^\gamma_{t+1}-\pi^\gamma\|_{TV}\leq 2\alpha^t
\end{align*}
with $\alpha=(1-P(B)\pi(\mathds{X})^N)<1$. Therefore, Assumption \ref{main_assmp} implies the geometric ergodicity of the vector state process of the team.
\end{remark}

\subsection{Existence of a Solution to the ACOE and Optimal Policies under Continuity of Transition and Cost Functions}\label{weak_finite}
We have shown the existence of a solution to the ACOE and thus the existence of optimal stationary Markov policies under a minorization condition for the dynamics of the agents. However, we have also observed that this approach may fail when the number of agents is large. 
 
In this section, we will show that a solution to the ACOE can be established using the vanishing discount approach under the assumption that the transition and cost functions of the agents are Lipschitz continuous. Furthermore, the approach presented in this section will be valid for large populations, and in particular for the infinite population problems as we will see later.
\begin{assumption}\label{lip_reg}
\begin{itemize}
\item[i.] $\mathds{X}$ and $\mathds{U}$ are compact.
\item[ii.] $f$ is Lipschitz in $x,u,\mu$ such that
\begin{align*}
|f(x,u,\mu,w^i,w^0)-f(x',u',\mu',w^i,w^0)|\leq K_f \left(|x-x'|+|u-u'|+W_1(\mu,\mu')\right)
\end{align*} 
for some $2K_f<1$, uniformly in $w^i, w^0$ where $W_1$ is the first order Wasserstein distance.
\item[iii.] $c$ is  Lipschitz in $x,u,\mu$ such that
\begin{align*}
|c(x,u,\mu)-c(x',u',\mu')|\leq K_c \left(|x-x'|+|u-u'|+W_1(\mu,\mu')\right)
\end{align*}
for some $K_c<\infty$.
\end{itemize}
\end{assumption}

\begin{theorem}\label{van_disc_fin}
Under Assumption \ref{lip_reg}, there exists a constant $j^*$, and a function $h\in \B(\P_N(\mathds{X}))$, such that 
\begin{align*}
j^*+h(\mu)=\inf_{\Theta\in U(\mu)}\bigg(k(\mu,\Theta)+\int_{\P_N(\mathds{X})} h(\mu_1)\eta(d\mu_1|\mu,\Theta)\bigg).
\end{align*}
Furthermore, there exists a stationary and Markov policy $g$ which achieves the infimum at the right hand side.
\end{theorem}

\begin{proof}
We start by analyzing the discounted cost criteria. Let
\begin{align*}
K_\beta^N(\mu,g)=\sum_{t=0}^{\infty}\beta^tE_{\mu}^g\left[k(\mu_t,\Theta_t)\right]
\end{align*}
be the discounted cost under some policy $g$ where $0<\beta<1$ is some discount factor. Furthermore, let $K_\beta^{N,*}(\mu)$ denote the optimal discounted cost function for the initial agent distribution $\mu\in\P_N(\mathds{X})$.

We define the following relative discounted value function, for some fixed $\mu'\in\P_N(\mathds{X})$:
\begin{align*}
h_\beta^N(\mu):=K^{N,*}_\beta(\mu)-K^{N,*}_\beta(\mu').
\end{align*}
Using the Bellman equation for $K^{N,*}_\beta(\mu)$ and by rearranging some terms, one can write that
\begin{align*}
K^{N,*}_\beta(\mu)-K^{N,*}_\beta(\mu')\\
=\inf_{\Theta\in U(\mu)}\bigg(k(\mu,\Theta)& + \beta \int\left( K_\beta^{N,*}(\mu_1)-K^{N,*}_\beta(\mu')\right)\eta(d\mu_1|\mu,\Theta)-(1-\beta)K^{N,*}_\beta(\mu')\bigg).
\end{align*}
Using the definition $h_\beta^N$ function, we can also write that
\begin{align}\label{rel_bell}
h_\beta^N(\mu)=\inf_{\Theta\in U(\mu)} \bigg(k(\mu,\Theta) + \beta  \int h_\beta^N(\mu_1)\eta(d\mu_1|\mu,\Theta)-(1-\beta)K^{N,*}_\beta(\mu')\bigg)
\end{align}
Next, we will show that $h_\beta^N$ is uniformly bounded and equicontinuous (over $\beta$) when we metrize  $\P_N(\mathds{X})$  with the first order Wasserstein distance. Thus we can invoke the Arzela-Ascoli Theorem since $\P_N(\mathds{X})$ is complete and separable under the Wasserstein metric.

For the boundedness:
\begin{align*}
\left|h_\beta^N(\mu)\right|=\left|K^{N,*}_\beta(\mu)-K^{N,*}_\beta(\mu')\right|\leq  \frac{2K_c}{1-2K_f\beta}W_1(\mu,\mu')
\end{align*}
where we used \cite[Lemma 4]{bayraktar2023finite} for the last step. Furthermore, since we assume that $\mathds{X}$ is compact (and bounded), $W_1(\mu,\mu')$ is uniformly bounded as well. Hence, $h_\beta^N(\mu)$ is uniformly bounded over $\beta$ when $2K_f<1$.

For equicontinuity, similarly, for some $\mu,\hat{\mu}$ we write
\begin{align}\label{fin_equi}
\left|h_\beta^N(\mu)-h_\beta^N(\hat{\mu})\right|=\left|K^{N,*}_\beta(\mu)-K^{N,*}_\beta(\mu')- K^{N,*}_\beta(\hat{\mu})+ K^{N,*}_\beta(\mu')\right|\leq  \frac{2K_c}{1-2K_f\beta}W_1(\mu,\hat{\mu})
\end{align}
which proves that $h_\beta^N(\mu)$ is equicontinuous over $\beta$ if $2K_f<1$, as the bound becomes uniform over $\beta$. 

Furthermore, the stage-wise cost function is uniformly bounded as $c$ is continuous and $\mathds{X}$ is compact. Hence, for the fixed $\mu'$, $(1-\beta)K^{N,*}_{\beta}(\mu')$ is a bounded sequence over $\beta$.  Thus, there exists a sequence of discount factors, say $\beta(k)\to 1$, such that, $(1-\beta(k))K_{\beta(k)}^{N,*}(\mu')\to j^*$ for some constant $j^*$. If we consider some other $\mu\in\P_N(\mathds{X})$, we can also show that
\begin{align*}
&|(1-\beta(k))K_{\beta(k)}^{N,*}(\mu)-j^*|\\
&\leq |(1-\beta(k))K_{\beta(k)}^{N,*}(\mu)-(1-\beta(k))K_{\beta(k)}^{N,*}(\mu')|+ |(1-\beta(k))K_{\beta(k)}^{N,*}(\mu')-j^*|\\
&\leq (1-\beta(k))\left|h^N_{\beta(k)}(\mu)\right| +|(1-\beta(k))K_{\beta(k)}^{N,*}(\mu')-j^*| \to 0
\end{align*}
where we used the boundedness of $h_\beta^N$ and the fact that along the sequence $\beta(k)$, $|(1-\beta(k))K_{\beta(k)}^{N,*}(\mu')-j^*| \to 0$. Hence, not only for the fixed $\mu'$ but for all $\mu\in\P_N(\mathds{X})$, we have that $|(1-\beta(k))K_{\beta(k)}^{N,*}(\mu)-j^*| \to 0$.

We have shown that $h_\beta^N$ is equicontinuous, hence there exists a further subsequence, say $\beta(k')$, such that $h^N_{\beta(k')}\to h$ uniformly, for some continuous function $h$, using the Arzela-Ascoli Theorem. If we take the limit of both sides in (\ref{rel_bell}) along $\beta(k')$, we can see
\begin{align*}
h(\mu)=\lim_{k' \to \infty}\inf_{\Theta\in U(\mu)} \bigg(k(\mu,\Theta) + \beta(k')  \int h_{\beta(k')}^N(\mu_1)\eta(d\mu_1|\mu,\Theta)\bigg) - j^*
\end{align*}

One can also justify the exchange of infimum and the limit using the equicontinuity of $h_\beta^N$, compactness of $\mathds{X}$, and thus the weakly pre-compactness of $\P_N(\mathds{X})$, see \cite[Theorem 7.3.3]{yuksel2020control}, once we exchange the limit and the infimum, we get
\begin{align*}
h(\mu)=\inf_{\Theta\in U(\mu)} \bigg(k(\mu,\Theta) +  \int h(\mu_1)\eta(d\mu_1|\mu,\Theta)\bigg) - j^*
\end{align*}
which proves the first part. The second part follows as the measurable selection conditions are satisfied under Assumption \ref{lip_reg}, since the limit function $h$, and the cost function $k$ are continuous, $\eta$ is weakly continuous and the action set $U(\mu)$ is compact and lower continuous.
\end{proof}

\adk{
The proof method used for the previous result reveals an important implication. Along any sequence (not just a subsequence) of vanishing discount factors, a normalized version of the discounted optimal value function converges to the average cost optimal value function.
\begin{proposition}\label{van_val}
Under Assumption \ref{lip_reg}, we have that for any $\mu\in\P_N(\mathds{X})$
\begin{align*}
\lim_{\beta\to 1} (1-\beta) K_\beta^{N,*}(\mu) = K_\infty^{N,*}(\mu)=j^*.
\end{align*}
\end{proposition}
\begin{proof}
Note that we know from the proof of Theorem \ref{van_disc_fin} that for any $\beta\to 1$, there exists a subsequence $\beta(k)\to 1$ such that 
\begin{align*}
\lim_{\beta(k)\to 1} (1-\beta(k)) K_\beta^{N,*}(\mu) =j^*
\end{align*}
where $j^*$ is the optimal value function for the average cost infinite horizon problem. 
The result then follows, since $j^*$, i.e. the optimal value function, is unique under Assumption \ref{lip_reg} and thus any convergent subsequence should converge to the same value.
\end{proof}}

\section{Ergodic Control of Mean Field Limit Problem}\label{inf_problem}
In Section \ref{meas_valued_sec}, we have focused on solutions to the problem of a finite population team, for which the resulting policy is a mapping $g:\P_N(\mathds{X})\to \P_N(\mathds{U}\times\mathds{X})$. In other words, the policy provides a distribution rule for the joint state-action of the agents, by looking at the state distribution of the agents. Note that this policy is a recipe for the team, however, to apply this policy at the agent level, one needs to coordinate the agents in order to realize the recipe state-action distribution. For example, if the distribution of the agents is given by $\mu\in\P_N(\mathds{X})$, and the optimal action is given by $\Theta(du,dx)\in\P_N(\mathds{U}\times\mathds{X})$, then $\Theta(du,dx)$ serves as a recipe for the team of agents. The agents need to coordinate to create this joint distribution among each other.
 In particular, every agent may be required to apply different, asymmetric policies, even if their states are the same, i.e. the agent level policies might depend on the identity of the agents. It might be the case for example, two agents $i\neq j$ with same state $x^i=x^j$, might need to apply different policies $\gamma^i(\cdot|x^i,\mu_{\bf x})\neq \gamma^j(\cdot|x^j,\mu_{\bf x})$, see Example \ref{counter_ex}. Asymmetric policies, especially for the large number of agents, can be hard to coordinate. 

 To deal with this coordination challenge, the usual approach is to consider the infinite population problem. The solution of the infinite population problem which we will introduce in this section, provides a policy, say $g:\P(\mathds{X})\to \P(\mathds{U}\times\mathds{X})$, that maps the marginal distribution of a representative agent to a joint state-action measure. For example, if the state is $\mu\in\P(\mathds{X})$, and the recipe action is some $\Theta(du,dx)\in \P(\mathds{U}\times\mathds{X})$, then by disintegrating, $\Theta(du,dx)=\gamma(du|x,\mu)\mu(dx)$, every agent can apply the agent level symmetric policy $\gamma(\cdot|x,\mu)$, which solves the coordination challenge.

In this section, we consider the control problem for the infinite population of agents, i.e. for $N\to \infty$. We will first provide a control problem for a single agent, and formulate it as a controlled measure valued Markov process.



\subsection{Infinite Population Measure Valued Control Problem}

We present a control problem for the measure process $\mu_t$, by changing the control variables as well. We will let the control actions be the joint measures on $\mathds{U}\times\mathds{X}$. Let $\Theta_t\in\P(\mathds{U}\times\mathds{X})$ such that the first marginal of $\Theta_t$ agrees with the state distribution of the agent $\mu_t$. We define the stage-wise cost function $k(\mu_t,\Theta_t)$ similar to the finite population measure valued MDP construction so that 
\begin{align*}
k(\mu_t,\Theta_t):=\int c(x,u,\mu_t)\Theta(du,dx).
\end{align*} 
An admissible policy $g=\{g_t\}$ is a sequence of control functions such that each can use the information variables $\{\mu_0,\dots,\mu_t,\Theta_0,\dots,\Theta_{t-1}\}$. Let $G$ be the set of all admissible policies for the measure valued process.

Recall the transition kernel $\mathcal{T}^{w^0}(\cdot|x,u,\mu)$ defined in (\ref{kernel_common_noise}), which defines a probability measure on $\mathds{X}$ given the last state, action, and measure variables $(x,u,\mu)$ and the common noise $w^0$. We can define the dynamics of the measure process $\mu_t$ using $\mathcal{T}^{w^0}$ so that
\begin{align}\label{inf_pop_meas_flow}
\mu_{t+1}(\cdot)=F(\mu_t,\Theta_t,w^0):=\int \mathcal{T}^{w^0}(\cdot|x,u,\mu_t)\Theta(dx,du).
\end{align}
Note that the dynamics can also be represented using a Markov kernel, which we will denote by $$\eta(d\mu_{t+1}|\mu_t,\Theta_t)$$ 
by an abuse of notation; recall that we have used the same notation to denote the transition kernel of the finite population measure valued MDP (see (\ref{trans_measure})). Even though these transition kernels are totally different, we use the same notation for notation simplicity.

We define the cost function to be minimized as
\begin{align*}
K_\infty(\mu,g)=\limsup_{T\to\infty}\frac{1}{T}\sum_{t=0}^{T-1}E_{\mu}^g\left[k(\mu_t,\Theta_t)\right]
\end{align*}
where $\mu=\mu_0$.

We denote the optimal cost function by $K^*_\infty(\mu)$ such that
\begin{align*}
K^*_\infty(\mu)=\inf_{g\in G}K_\infty(\mu,g).
\end{align*}

\begin{remark}\label{caution_remark}
We note that even though the infinite population problem can be structured as a classical single agent MDP problem, the standard tools used to study the average cost optimality criteria for MDPs may not be directly applicable. The standard methods (see e.g. \cite{HernandezLermaMCP,hernandezlasserre1999further,survey}) requires certain ergodicity conditions on the controlled state process, as well as mixing type conditions of the state process under total variation norm, e.g. conditions similar to Assumption \ref{main_assmp}. These conditions may be easy to verify for state variables living in general Euclidean state  spaces, e.g. for the finite population problem as we have seen in Section \ref{fin_rel_op}. However, for the measure valued controlled process we have introduced in this section for the infinite population problem, these conditions are not applicable unless we assume very strong conditions on the system dynamics.  For example, a slightly weaker version of Assumption \ref{main_assmp} requires that
\begin{align}\label{erg_cont}
\|\eta(\cdot|\mu,\Theta)-\eta(\cdot|\mu',\Theta')\|_{TV}\leq 2\alpha
\end{align}
for some $\alpha<1$.
If we assume for simplicity that there is no common noise and consider the infinite population measure valued process, this inequality will always fail, since the process becomes deterministic, and in particular we get $\|\eta(\cdot|\mu,\Theta)-\eta(\cdot|\mu',\Theta')\|_{TV} =2$
for any $\mu\neq\mu'$ and  $\Theta\neq\Theta'$. In general, most of the tools developed to study MDPs under average cost criteria may not be applicable for deterministic systems.

With the presence of common noise, the dynamics becomes stochastic again, however, verifying the standard average cost criteria assumptions is not immediate. For example, as we have observed in Remark \ref{inf_fails}, in order to verify (\ref{erg_cont}) for infinite population dynamics, a sufficient condition is that there exists a measurable set $B$ with positive probability, $P(B)>0$, such that for any $x,u,\mu$ we have 
\begin{align*}
\mathcal{T}^{w^0}(\cdot|x,u,\mu)=\pi(\cdot)
\end{align*}
for all $w^0\in B$, where $\pi$ is a probability measure. In words, there exists a set of common noise realizations which resets the measure flow to some $\pi(\cdot)\in\P(\mathds{X})$. Obviously, this is a very restrictive assumption, and thus, in what follows, we use different tools to study the average cost optimality. 
\end{remark}

\subsection{Existence of a Solution to the ACOE and Optimal Policies under Continuity of Transition and Cost Functions}\label{weak_inf}

We will first analyze the optimality problem, and establish the existence of stationary Markov optimal policies for the infinite population problem as in the finite population problem. We follow the same steps in Section \ref{weak_finite} and establish a result parallel to Theorem \ref{van_disc_fin}. As before, we will use the vanishing discount approach. Let
\begin{align*}
K_\beta(\mu,g)=\sum_{t=0}^{\infty}\beta^tE_{\mu}^g\left[k(\mu_t,\Theta_t)\right]
\end{align*}
where $0<\beta<1$ is some discount factor. Furthermore, let $K_\beta^*(\mu)$ denote the optimal discounted cost function.

We define the following relative discounted value function, for some fixed $\mu'\in\P(\mathds{X})$:
\begin{align*}
h_\beta(\mu):=K^*_\beta(\mu)-K^*_\beta(\mu').
\end{align*}


\begin{theorem}\label{van_disc_inf}
Under Assumption \ref{lip_reg}, there exists a constant $j^*$, and a function $h\in \B(\P_N(\mathds{X}))$, such that 
\begin{align*}
j^*+h(\mu)=\inf_{\Theta\in U(\mu)}\bigg(k(\mu,\Theta)+\int_{\P(\mathds{X})} h(\mu_1)\eta(d\mu_1|\mu,\Theta)\bigg).
\end{align*}
Furthermore, there exists a stationary and Markov policy $g$ which achieves the infimum at the right hand side.
\end{theorem}

\begin{proof}
As in the proof of Theorem \ref{van_disc_fin}, we define the relative discounted value function for some fixed $\mu'\in\P(\mathds{X})$:
\begin{align*}
h_\beta(\mu)=K_\beta^*(\mu)-K_\beta^*(\mu').
\end{align*}
A careful look at the proof of Theorem \ref{van_disc_fin} shows that it is sufficient to show $h_\beta$ is uniformly bounded and equicontinuous over $\beta\in(0,1)$.

For $h_\beta$, we have that
\begin{align}\label{inf_equi}
\left|h_\beta(\mu)\right|=\left|K_\beta^*(\mu)-K_\beta^*(\mu')\right|\leq& \lim_{N\to\infty}\left|K_\beta^*(\mu)-K_\beta^{N,*}(\mu^N)\right|+ \left|K_\beta^{N,*}(\mu^N) - K_\beta^{N,*}(\mu^{N'})\right|\nonumber\\
& +  \left|K_\beta^{N,*}(\mu^{N'}) - K_\beta^*(\mu')\right|\nonumber\\
&\leq \frac{2K_c}{1-2K_f\beta}W_1(\mu,\mu')
\end{align}
where $\mu^N$ and $\mu^{N'}$ are chosen such that $\mu^N\to\mu$ and $\mu^{N'}\to\mu'$ weakly. Taking the limit $N\to\infty$, the first and last terms go to $0$ by \cite{bauerle2021mean,motte2022mean} under Assumption \ref{lip_reg}, and the middle term is bounded by $ \frac{2K_c}{1-2K_f\beta}W_1(\mu^N,\mu^{N'})$ as we have shown in (\ref{fin_equi}). Hence, we have uniform boundedness together with compactness of $\mathds{X}$. For the equicontinuity one can follow the same steps.

For the second part of the result, we have that the measurable selections conditions holds under Assumption \ref{lip_reg}, as the cost function $k(\mu,\Theta)$ is continuous in both variables, and the transition kernel $\eta(\cdot|\mu,\Theta)$ is weakly continuous in both variables, furthermore $U(\mu)$ is compact and lower-semi continuous. Therefore, the proof is complete.
\end{proof}

\adk{
We now present a result, similar to Proposition \ref{van_val}. The result is about the limit of normalized discounted optimal value functions. The proof is identical to the proof of  Proposition \ref{van_val} and follows from the uniqueness of the optimal value function for the average cost criteria.
\begin{proposition}\label{van_disc_seq}
Under Assumption \ref{lip_reg}, we have that for any $\mu\in\P(\mathds{X})$
\begin{align*}
\lim_{\beta\to 1} (1-\beta) K_\beta^{*}(\mu) = K_\infty^{*}(\mu)=j^*.
\end{align*}
\end{proposition}
}

\section{Limit Theorems for $N\to\infty$}\label{N_to_inf}
Finally, in this section, we study the relation between the finite population control and the infinite population control. In particular, we will show that the value function of the $N$- population problem converges to the value function of the infinite population problem. We will show that the accumulation points of the optimal state-action distributions for the $N$-population problem are optimal state-action distribution for the infinite population problem. Furthermore, we will establish that one can symmetrically use policies designed for the infinite population problem for the finite population control with near optimality if the population is sufficiently large.

\subsection{Convergence of Value Functions as $N\to \infty$}\label{val_conv_inf}



\begin{theorem}\label{conv_thm}
Under Assumption \ref{lip_reg}
\begin{align*}
\lim_{N\to\infty}K_\infty^{N,*}(\mu^N)=K_\infty^*(\mu)
\end{align*}

\end{theorem}

\begin{proof}
We will prove the result using the ACOE for the finite population and the infinite population problem. 
Recall that we have proved the existence of a solution to the ACOE for the finite and infinite population problems under Assumption \ref{lip_reg}, in Theorem \ref{van_disc_inf} and Theorem \ref{van_disc_fin} such that:
\begin{align*}
&h^N(\mu^N) = \inf_{\Theta^N}\left\{k(\mu^N,\Theta^N)+ \int h^N(\mu_1^N) \eta(d\mu_1^N|\mu^N,\Theta^N) \right\} - j^N\\
&h(\mu)=\inf_{\Theta}\left\{k(\mu,\Theta)+ \int h(\mu_1) \eta(d\mu_1|\mu,\Theta) \right\} - j
\end{align*}
for some measurable relative value functions $h^N, h$, and where $j^N, j$ are constants that are equal to the optimal value functions, i.e. $K_\infty^{N,*}(\mu^N)=j^N$ for all $\mu^N$ and  $K_\infty^{*}(\mu)=j$ for all $\mu$.
In particular, under Assumption \ref{lip_reg}, there exist optimal selectors, say $f$ and $f^N$ such that
\begin{align*}
&h^N(\mu^N) = k(\mu^N,f^N(\mu^N))+ \int h^N(\mu_1^N) \eta(d\mu_1^N|\mu^N,f^N(\mu^N))  - j^N\\
&h(\mu)=k(\mu,f(\mu))+ \int h(\mu_1) \eta(d\mu_1|\mu,f(\mu)) - j
\end{align*}

Iterating these equalities $T$ times we get
\begin{align}\label{acoe_bound}
&h^N(\mu^N)= \sum_{t=0}^{T-1}E_{\mu^N}^{f^N}\left[k(\mu_t^N,\Theta^N_t)\right]+ E_{\mu^N}^{f^N}\left[h^N(\mu_T^N)\right]-Tj^N\nonumber\\
&h(\mu)= \sum_{t=0}^{T-1}E_{\mu}^{f}\left[k(\mu_t,\Theta_t)\right]+ E_{\mu}^{f}\left[h(\mu_T)\right]-Tj
\end{align}

If we denote the finite horizon costs under some policies $g^N,g$ by
\begin{align*}
&K_T^N(\mu^N,g^N)=E_{\mu^N}^{g^N}\left[\sum_{t=0}^{T-1}k(\mu_t^N,\Theta^N_t)\right]\\
&K_T(\mu,g)=E_{\mu}^{g}\left[\sum_{t=0}^{T-1}k(\mu_t,\Theta_t)\right]
\end{align*}
We can then write, using (\ref{acoe_bound})
\begin{align*}
&K_T^N(\mu^N,f^N)-Tj^N=h^N(\mu^N)- E_{\mu^N}^{f^N}\left[h^N(\mu_T^N)\right]\\
&K_T(\mu,f)-Tj=h(\mu)- E_{\mu}^{f}\left[h(\mu_T)\right]
\end{align*}

We will later show that $\|h\|_\infty\leq M$ and $\|h^N\|_\infty\leq M$, where $M$ is uniform over $N$. Hence, we can write
\begin{align*}
&\frac{K_T^N(\mu^N,f^N)}{T}-j^N\leq \frac{2M}{T}\\
&\frac{K_T(\mu,f)}{T}-j\leq\frac{2M}{T}
\end{align*}

If we iterate the ACOE's, over some policies $g,g^N$ (including the finite horizon optimal policies), instead of the selectors $f,f^N$, the equations in (\ref{acoe_bound}), can be represented as the following inequalities:
\begin{align*}
&h^N(\mu^N)\leq \sum_{t=0}^{T-1}E_{\mu^N}^{g^N}\left[k(\mu_t^N,\Theta^N_t)\right]+ E_{\mu_0^N}^{g^N}\left[h^N(\mu_T^N)\right]-Tj^N\nonumber\\
&h(\mu)\leq \sum_{t=0}^{T-1}E_{\mu}^{g}\left[k(\mu_t,\Theta_t)\right]+ E_{\mu}^{g}\left[h(\mu_T)\right]-Tj
\end{align*}
Using the same bound on $h,h^N$
\begin{align*}
&\frac{-2M}{T}\leq\frac{K_T^N(\mu^N,g^N)}{T}-j^N\\\
&\frac{-2M}{T}\leq \frac{K_T(\mu,g)}{T}-j
\end{align*}
As we have noted, the above lower bounds are valid for the finite horizon optimal policies as well. Furthermore, because of the optimality, we have that $\frac{K_T^{N,*}(\mu^N)}{T}\leq \frac{K_T^{N}(\mu^N,f^N)}{T}$ and $\frac{K_T^{*}(\mu)}{T}\leq \frac{K_T(\mu,f)}{T}$ where $f$ and $f^N$ are the selectors for the ACOE, i.e. the optimal policies for the infinite horizon average cost problem. Hence, we can write
\begin{align}\label{unif_over_N}
&\frac{-2M}{T}\leq\frac{K_T^{N,*}(\mu^N)}{T}-j^N\leq \frac{K_T^{N}(\mu^N,f^N)}{T}-j^N\leq \frac{2M}{T}\\
&\frac{-2M}{T}\leq \frac{K_T^*(\mu)}{T}-j\leq \frac{K_T(\mu,f)}{T}-j\leq \frac{2M}{T}\nonumber
\end{align}
Hence, if $M$ is uniform over $N$, taking $T\to\infty$, we get
\begin{align}\label{finite_to_inf}
&\frac{K_T^{N,*}(\mu^N)}{T}-j^N \to 0\nonumber\\
&\frac{K_T^*(\mu)}{T}-j\to 0
\end{align}
uniformly over $N$.

To show that $h$ and $h^N$ are uniformly bounded over $N$, we write using (\ref{fin_equi}) and ($\ref{inf_equi}$)
\begin{align*}
&h^N(\mu^N)=\lim_{\beta\to 1} h_\beta^N (\mu^N)\leq M \\
&h(\mu)=\lim_{\beta\to 1} h_\beta (\mu)\leq M 
\end{align*} 
where $M$ is independent of $N$.

Finally, we write
\begin{align*}
&\left|K_\infty^{N,*}(\mu^N)-K_\infty^*(\mu)\right|= \left|K_\infty^{N,*}(\mu^N)-K_\infty^*(\mu)\pm \frac{K_T^*(\mu)}{T}\pm \frac{K_T^{N,*}(\mu^N)}{T}\right|\\
&\leq \left|K_\infty^{N,*}(\mu^N)- \frac{K_T^{N,*}(\mu^N)}{T}\right|+ \left|\frac{K_T^{N,*}(\mu^N)}{T}-  \frac{K_T^{*}(\mu)}{T}\right|\\
&\quad + \left| \frac{K_T^{*}(\mu)}{T} - K_\infty^*(\mu)\right|
\end{align*}
We have shown in (\ref{finite_to_inf}) that we can choose a large enough $T$, independent of $N$ such that the first and the last terms are less than $\epsilon/3$ for any given $\epsilon>0$. For the chosen $\epsilon$ and $T$, it has been shown in \cite{bauerle2021mean,motte2022mean} that the middle term goes to $0$, i.e. can be made less than $\epsilon/3$, for some large $N<\infty$, under Assumption \ref{lip_reg}. Hence the proof is completed.

\end{proof}

\adk{
\subsection{Accumulation to Optimal Solutions for Infinite Populations}\label{accum}
We have established the existence of an optimal policy for the $N$-population control problem, say $g^N:\P_N(\mathds{X})\to\P_N(\mathds{U}\times\mathds{X})$. Under this optimal policy, we have that 
\begin{align*}
K^N_\infty(\mu^N,g^N)=\lim_{T\to\infty}\frac{1}{T}\sum_{t=0}^{T-1}E^{g^N}\left[\int c(x,u,\mu_t^N)\Theta_t^N(du,dx)\right]=j^N
\end{align*}
for all $\mu^N$ where $j^N$ is the optimal value of the problem. The dynamics for the state-action and state distributions are given in (\ref{trans_measure}). We can also rewrite the accumulated cost under the optimal policy as
\begin{align*}
K^N_\infty(\mu^N,g^N)&=\lim_{T\to\infty}\frac{1}{T}\sum_{t=0}^{T-1}E^{g^N}\left[\int c(x,u,\mu_t^N)\Theta_t^N(du,dx)\right]\\
&=\lim_{T\to\infty}\frac{1}{T}\sum_{t=0}^{T-1} \int k(\mu,\Theta)P_t^N(d\Theta)
\end{align*}
where $P_t^N(d\Theta)\in \P(\P_N(\mathds{U}\times\mathds{X}))\subset  \P(\P(\mathds{U}\times\mathds{X}))$ is the marginal distribution for the state-action distribution of the agents at time $t$ under the policy $g^N$. Note also that $\mu$ is determined by $\Theta$ as it is simply the second marginal of $\Theta(du,dx)$. The next result shows that the accumulation points of $P_t^N$ coincides with the optimal flow of the infinite population control problem as $N$ grows.

We first present a useful lemma.
\begin{lemma}\label{key_lemma}
Assume Assumption \ref{lip_reg} holds. Consider a sequence of functions $\{g^N\}_N$ defined on $\P_N(\mathds{X})$ such that 
\begin{align*}
g^N(\mu^N)\to g(\mu)
\end{align*}
for some function $g$ and for any $W_1(\mu^N,\mu)\to 0$. Assume further that
\begin{align*}
\left|g^N(\mu^N)- g^N(\hat{\mu}^N)\right|\leq K W_1(\mu^N,\hat{\mu}^N)
\end{align*}
for all $\mu^N,\hat{\mu}^N\in\P_N(\mathds{X})$ and for some $K<\infty$ that is independent of $N$.

We then have for any $\mu^N \in \P_N(\mathds{X}), \Theta^N\in \P_N(\mathds{U}\times\mathds{X})$ such that $W_1(\mu^N,\mu)\to 0$ and $W_1(\Theta^N,\Theta)\to 0$ for some $\mu\in\P(\mathds{X})$ and $\Theta\in\P(\mathds{U}\times\mathds{X})$
\begin{align*}
\lim_{N\to\infty}\int g^N(\mu_1^N)\eta^N(d\mu_1^N|\mu^N,\Theta^N) = \int g(\mu_1)\eta(d\mu_1|\mu,\Theta).
\end{align*}
\end{lemma}
\begin{proof}
The proof can be found in Appendix \ref{key_lemma_proof}.
\end{proof}

\begin{theorem}\label{accu}
Under Assumption \ref{lip_reg}, we can find a subsequence, say $N'$ such that $P_t^{N'} \to P_t$ weakly for every $t$ for some $P_t\in\P(\P(\mathds{U}\times\mathds{X}))$. Furthermore, the limit flow $\{P_t\}$ is the optimal flow for the infinite population control problem such that 
\begin{align*}
\lim_{T\to\infty}\frac{1}{T}\int k(\mu,\Theta)P_t(d\Theta) = K_\infty^*(\mu)= j 
\end{align*}
\end{theorem}
\begin{proof}
 We first note that since $\mathds{U}\times\mathds{X}$ is assumed to be compact and Polish, $\P(\mathds{U}\times\mathds{X})$ is also a compact and Polish under the weak convergence topology. In turn, $\P(\P(\mathds{U}\times\mathds{X}))$ is also compact and Polish.

 Hence, we can find a sebsequence $N_0$ such that $P_0^{N_0}\to P_0$ weakly for some $P_0 \in \P(\P(\mathds{U}\times\mathds{X}))$. In fact, for time $t=0$, the common noise does not affect the distributions. Thus, if the initial state distribution $\mu^N\to \mu$, then one can show that $P_0(\cdot)= \delta_{\Theta_0}(\cdot)$ for some $\Theta_0\in\P(\mathds{U}\times\mathds{X})$ such that $\Theta_0(\mathds{U},\cdot)=\mu(\cdot)$.

For $P_1^{N_0}$, one can find a further subsequence $N_1$, such that $P_1^{N_1}\to P_1$ for some $P_1\in  \P(\P(\mathds{U}\times\mathds{X}))$. Continuing in this manner, and using a standard diagonal argument, we can find a subsequence $N'$ such that $P_t^{N'}\to P_t$ weakly for all $t\geq 0$.

We now show that the limit flow $\{P_t\}_t$ is consistent with the dynamics of the infinite population dynamics. Note that any $P_t\in\P(\P(\mathds{U}\times\mathds{X}))$ induces a probability measure in $\P(\P(\mathds{X}))$, say $\bar{P}_t$ such that
\begin{align*}
\bar{P}_t(\mu\in A) = P_t(\{\Theta: \Theta(\mathds{U},\cdot)\in A\})
\end{align*} 
for any $A\in \B(\P(\mathds{X}))$.
The sequence $\{P_t\}\subset \P(\P(\mathds{U}\times\mathds{X}))$ is consistent with the infinite population dynamics, if we have
\begin{align*}
\bar{P}_{t+1}(\mu_{t+1}\in A) = \int \eta(\mu_{t+1}\in A|\mu,\Theta)P_t(d\Theta).
\end{align*}
For the $N$- population dynamics, we have that 
\begin{align}\label{finite_dyn}
\int f(\mu_{t+1})\bar{P}^N_{t+1}(d\mu_{t+1}) = \int f(\mu_{t+1})\eta^N(d\mu_{t+1}|\mu,\Theta)P_t^N(d\Theta)
\end{align}
for any continuous and bounded function $f$, where we use $\eta^N$ to denote the one step transition kernel of the $N$-population dynamics.
By Lemma \ref{key_lemma}, for any $\mu^N,\Theta^N \to \mu,\Theta$, we can write
\begin{align*}
\int f(\mu_{t+1})\eta^N(d\mu_{t+1}|\mu^N,\Theta^N) \to \int f(\mu_{t+1})\eta(d\mu_{t+1}|\mu,\Theta).
\end{align*}
Hence,  if we take the limit of both sides in (\ref{finite_dyn}) along the chosen subsquence $N'$, we get 
\begin{align*}
\int f(\mu_{t+1})\bar{P}_{t+1}(d\mu_{t+1}) = \int f(\mu_{t+1})\eta(d\mu_{t+1}|\mu,\Theta)P_t(d\Theta)
\end{align*}
where we used the fact that $P_t^{N'}\to P_t$ weakly for all $t$ and \cite[Theorem 3.1]{Lan81}. Since the class of continuous and bounded functions are measure determining (see \cite[Theorem 1.2]{Billingsley}), we can conclude that the limit sequence $\{P_t\}_t$ is consistent with the infinite population dynamics.

We now prove the second part of the result. The result is an implication of Theorem \ref{conv_thm}. By (\ref{unif_over_N}), we have that 
\begin{align*}
\frac{K_T^{N}(\mu^N,g^N)}{T}\to j^N
\end{align*}
where $g^N$ is the optimal stationary policy for the $N$-population. Furthermore, the convergence is uniform over $N$. Thus, if we take the limit of both sides along the subsequence $N'$, we get
\begin{align*}
&\lim_{N'\to\infty}\lim_{T\to\infty}\frac{K_T^{N'}(\mu^{N'},g^{N'})}{T}=\lim_{N'\to\infty} j^{N'}=j\\
&=\lim_{T\to\infty} \lim_{N'\to\infty}\frac{K_T^{N'}(\mu^{N'},g^{N'})}{T} =\lim_{T\to\infty}\frac{1}{T}\sum_{t=0}^{T-1} \lim_{N'\to\infty}\int k(\mu,\Theta)P_t^{N'}(d\Theta)\\
&=\lim_{T\to\infty}\frac{1}{T}\sum_{t=0}^{T-1}\int k(\mu,\Theta)P_t(d\Theta)
\end{align*}
where we used the weak convergence of $P_t^{N'}\to P_t$ and the continuity of the function  $k(\mu,\Theta)$ on $(\mu,\Theta)$ under Assumption \ref{lip_reg}.
This shows that
\begin{align*}
K_\infty^*(\mu)=j=\lim_{T\to\infty}\frac{1}{T}\sum_{t=0}^{T-1} \int k(\mu,\Theta)P_t(d\Theta).
\end{align*}
\end{proof}
}

\adk{
\subsection{Near Optimality of Symmetric Policies for Finite Population Control}\label{sym_pol_sec} In this section, we focus on the effect of using symmetric policies for finite population control problem. The following example shows that the symmetric policies may not achieve the optimal performance, and personalized policies have to be used for the optimality.
\begin{example}\label{counter_ex}
Consider a team control problem with two agents, i.e. $N=2$. We assume that $\mathds{X}=\mathds{U}=\{0,1\}$. The stage wise cost function of the agents is defined as
\begin{align*}
c(x,u,\mu_{\bf x}) = W_1(\mu_{\bf x},\bar{\mu})
\end{align*} 
where 
\begin{align*}
\bar{\mu}=\frac{1}{2}\delta_0 + \frac{1}{2}\delta_1.
\end{align*}
In words, the state distribution should be distributed equally over the state space $\{0,1\}$ for minimal stage-wise cost. For the dynamics we assume a deterministic model such that
\begin{align*}
x_{t+1}=u_t.
\end{align*}
In words, the action of an agent purely determines the next state of the same agent. The goal of the agents is to minimize
\begin{align*}
K_\infty(x_0^1,x_0^2,g^1,g^2)=\limsup_{T\to \infty} \frac{1}{T}\sum_{t=0}^{T-1}E^{g^1,g^2}\left[\frac{c(x_t^1,u_t^1,\mu_{\bf x_t}) +c(x_t^2,u_t^2,\mu_{\bf x_t})  }{2}\right]
\end{align*}
for some initial state values ${\bf{x}}_0=[x_0^1,x_0^2]$, by choosing policies $g^1,g^2$. The expectation is over the possible randomization of the policies.  We assume full information sharing such that every agent has access to the state and action information of the other agent. 

We let the initial states be $x_0^1=x_0^2=0$. An optimal policy for the agents for the problem is given by
\begin{align*}
&g^1(0,0)= 0, \qquad g^2(0,0)=1\\
&g^1(0,1)= 0 ,\qquad g^2(0,1)=1\\
 &g^1(1,0)= 1 ,\qquad g^2(1,0)=0\\
&g^1(1,1)= 1 ,\qquad g^2(1,1)=0
\end{align*}
which always spreads the agents equally over the state space. One can realize that, when the agents are positioned at either $(0,0)$ or $(1,1)$, they have to use personalized policies to decide on which one to be placed at $0$ or $1$. 

For any symmetric policy $g^1(x^1,x^2)=g^2(x^1,x^2)=g(x^1,x^2)$, including the randomized ones, there will always be cases with strict positive probability, where the agents are positioned at the same state, and thus the performance will be strictly worse than the optimal performance.
\end{example}

We now introduce the symmetric  policies that will be used by the agents for the finite population setting. We will focus on the optimal policies for infinite population under the discounted cost criteria. Let $\Theta(du,dx)$ be an optimal state-action distribution  for some measure $\mu\in\P(\mathds{X})$. We then write
\begin{align*}
\Theta(du,dx)=\gamma(du|x,\mu)\mu(dx).
\end{align*}
In the finite population setting, we let the agents use the symmetric randomized policy $\gamma(du|x,\mu)$. That is, an agent $i$, at time $t$, observes their local state $x_t^i$ and the state distribution of the agents, say $\mu_t^N$, and applies $\gamma(du|x^i_t,\mu_t^N)$. This agent level policy then defines a team policy for the state distribution $\mu^N$ such that 
\begin{align*}
g_\beta(\mu^N):=\Theta^N(du,dx)=\gamma(du|x,\mu^N)\mu^N(dx).
\end{align*}
We use the $\beta$ dependence to emphasize that the policy is constructed using the discounted cost criteria.

Our next result shows that this policy, which is symmetric between the agents, will achieve near optimal performance for sufficiently large populations under the discounted cost criteria.
\begin{proposition}\label{sym_pol_disc}
Under Assumption \ref{lip_reg}, for any $\mu^N\to \mu$
\begin{align*}
\lim_{N\to\infty}K_\beta^N(\mu^N,g_\beta)= K_\beta^*(\mu).
\end{align*}
Furthermore,
\begin{align*}
\lim_{N\to\infty} K_\beta^N(\mu^N,g_\beta)-K_\beta^{N,*}(\mu^N) = 0.
\end{align*}
Hence, $g_\beta$ will achieve near optimal performance for finite populations.
\end{proposition}

\begin{proof}
The proof can be found in Appendix \ref{sym_pol_disc_proof}.
\end{proof}

\begin{theorem}
Under Assumption \ref{lip_reg}, for any given $\epsilon>0$ we can find a $\beta<1$ sufficiently close to 1, and some $\bar{N}$ such that for all $N>\bar{N}$ we have that 
\begin{align*}
\left|K_\infty^N(\mu^N,g_\beta) - j^N\right| \leq \epsilon.
\end{align*}
In words, the symmetric feedback policy $g_\beta$ is near optimal for $N$-population problem if $N$ is sufficiently large.
\end{theorem}
\begin{proof}
Recall $h_\beta(\mu):=K_\beta^*(\mu)-K_\beta^*(\mu_0)$. Using the Bellman equation for $K_\beta^*(\mu)$ we can write
\begin{align}\label{eq1}
h_\beta(\mu) &= k(\mu,\Theta) + \beta\int h_\beta(\mu_1)\eta(d\mu_1|\mu,\Theta) - (1-\beta)K_\beta^*(\mu_0)\nonumber\\
&= k(\mu,\Theta) + \int h_\beta(\mu_1)\eta(d\mu_1|\mu,\Theta) -  (1-\beta)\int h_\beta(\mu_1)\eta(d\mu_1|\mu,\Theta)- (1-\beta)K_\beta^*(\mu_0)
\end{align}
where $\Theta$ is an optimal state-action distribution for $\mu$. We also recall $h_\beta^N(\mu^N):=K_\beta^{N,*}(\mu^N)-K_\beta^{N,*}(\mu_0^N)$ where $\mu_0^N\in\P_N(\mathds{X})$ is chosen such that $\mu_0^N \to \mu_0$.  We can then rewrite (\ref{eq1}) as
\begin{align*}
\pm h_\beta^N(\mu^N)\pm j^N + h_\beta(\mu) =&  k(\mu,\Theta) + \int h_\beta(\mu_1)\eta(d\mu_1|\mu,\Theta) \\
&-  (1-\beta)\int h_\beta(\mu_1)\eta(d\mu_1|\mu,\Theta)- (1-\beta)K_\beta^*(\mu_0)\\
& \pm \left(k(\mu^N,g_\beta(\mu^N)) + \int h_\beta^N(\mu_1^N)\eta(d\mu_1^N|\mu^N,g_\beta(\mu^N)) \right).
\end{align*}
 By rearranging the terms we can write
\begin{align}
&k(\mu^N,g_\beta^\infty(\mu^N)) + \int h_\beta^N(\mu_1^N)\eta(d\mu_1^N|\mu^N,g_\beta^\infty(\mu^N))\nonumber\\
=& j^N + h_\beta^N(\mu^N)\nonumber\\
& + h_\beta(\mu)-h_\beta^N(\mu^N)\label{1}\\
&+ j - j^N\label{2}\\
&  -\left(j - (1-\beta)K_\beta^*(\mu_0)\right)  + (1- \beta)\int h_\beta(\mu_1)\eta(d\mu_1|\mu,\Theta)\label{3}\\
&+ k(\mu^N,g_\beta^\infty(\mu^N)) - k(\mu,\Theta)\label{4}\\
&+\int h_\beta^N(\mu_1^N)\eta(d\mu_1^N|\mu^N,g_\beta^\infty(\mu^N)) - \int h_\beta(\mu_1)\eta(d\mu_1|\mu,\Theta)\label{5}
\end{align}
We first choose $\beta$ large enough (independent of $N$) such that (\ref{3}) is smaller than $\epsilon$. This can be done using Proposition \ref{van_disc_seq}, and since $h_\beta$ is uniformly bounded over $\beta$ (see (\ref{inf_equi})).

For the chosen $\beta$, we can make (\ref{1}) smaller than $\epsilon$, since $K_\beta^{N,*}(\mu)\to K_\beta^*(\mu)$ and $K_\beta^{N,*}(\mu_0^N)\to K_\beta^*(\mu_0)$ as $N\to \infty$ for fixed $\beta<1$. (\ref{2}) can be made smaller than $\epsilon$ by Theorem \ref{conv_thm}.

Finally we focus on  (\ref{4})  and (\ref{5}) :
\begin{align*}
&k(\mu^N,g_\beta(\mu^N) +\int h_\beta^N(\mu_1^N)\eta(d\mu_1^N|\mu^N,g_\beta(\mu^N))  - k(\mu,\Theta)-\int h_\beta(\mu_1)\eta(d\mu_1|\mu,\Theta)\\
&=k(\mu^N,g_\beta(\mu^N) + \beta \int h_\beta^N(\mu_1^N)\eta(d\mu_1^N|\mu^N,g_\beta(\mu^N)) + (1-\beta) \int h_\beta^N(\mu_1^N)\eta(d\mu_1^N,\mu^N,g_\beta(\mu^N))\\
&\quad - k(\mu,\Theta)-\beta \int h_\beta(\mu_1)\eta(d\mu_1|\mu,\Theta) - (1-\beta) \int h_\beta(\mu_1)\eta(d\mu_1|\mu,\Theta) \\
& =k(\mu^N,g_\beta(\mu^N) + \beta \int K_\beta^{N,*}(\mu_1^N)\eta(d\mu_1^N|\mu^N,g_\beta(\mu^N))  - K_\beta^*(\mu) - \beta K_\beta^{N,*}(\mu_0^N) + \beta K_\beta^*(\mu_0)\\
&\quad  +(1-\beta) \int h_\beta^N(\mu_1^N)\eta(d\mu_1^N|\mu^N,g_\beta(\mu^N)) -  (1-\beta) \int h_\beta(\mu_1)\eta(d\mu_1|\mu,\Theta)
\end{align*}
The last line can be made smaller than $\epsilon$ by choosing $\beta$ sufficiently close to $1$ independent of $N$ since $h_\beta^N$ and $h_\beta$ are previously shown to be uniformly bounded (uniform over $N$ as well). Furthermore, for the chosen $\beta$, we have that $\beta K_\beta^{N,*}(\mu_0^N) \to \beta K_\beta^*(\mu_0)$ and finally we have shown in the proof of Proposition \ref{sym_pol_disc} that 
\begin{align*}
\lim_{N\to \infty}k(\mu^N,g_\beta(\mu^N) + \beta \int K_\beta^{N,*}(\mu_1^N)\eta(d\mu_1^N|\mu^N,g_\beta(\mu^N)) = K_\beta^*(\mu).
\end{align*}
We then can make the terms  (\ref{4})  and (\ref{5}) smaller than $\epsilon$ by choosing $\beta$ and $N$ sufficiently large.

 Hence we can write the following for the chosen $\beta$ and $N$
\begin{align*}
k(\mu^N,g_\beta^\infty(\mu^N)) + \int h_\beta^N\eta(d\mu_1^N|\mu^N,g_\beta^\infty(\mu^N)) \leq j^N + 5\epsilon + h_\beta^N(\mu^N).
\end{align*}
We note the following inequality to conclude the result (see \cite[Theorem 7.1.3]{yuksel2020control}): if
\begin{align*}
g+h(\mu^N)\geq k(\mu^N,f(\mu^N)) + \int h(\mu_1^N)\eta(d\mu_1^N|\mu^N,f(\mu^N))
\end{align*}
for some functions $h,f$ and for some constant $g$ and  if $h$ is bounded then
\begin{align*}
g\geq K_\infty^N(\mu^N,f).
\end{align*}
We have established this inequality with $g=j^N+5\epsilon$, $f\equiv g_\beta$ and $h\equiv h_\beta^N$. Therefore we can conclude that
\begin{align*}
K_\infty^N(\mu^N,g_\beta)\leq j^N + 5\epsilon
\end{align*}
as $h^N_\beta$ is bounded uniformly.
\end{proof}

\begin{remark}
The symmetric policy we have considered in this section requires an agent to have access to two main source of information. At every time step $t$, the agent $i$ needs to have access to the local state $x_t^i$ and the distribution of the state variables of the other agents, i.e. the mean-field term $\mu_t^N=\mu_{\bf x_t}$. Although the computation of the policies do not require coordination, to execute the policies the agents need to coordinate to collect the mean-field term, or a coordinator needs to provide the mean-field term to the agents. Thus, these policies are not decentralized to be precise.

An alternative way to use the infinite population limit solution for the finite population control, is by estimating the marginal distribution of  the local state variable without communicating the other agents. If, for example, there is no common noise for the dynamics, agent $i$ might simply consider the mean-filed term for the infinite population problem $\mu_t$ given by the dynamics (\ref{inf_pop_meas_flow}). Without common noise, this process forms a deterministic sequence, and thus the agent $i$ can estimate this measure without communicating with other agents. If there is a common noise, then the agents can estimate this fictitious infinite population mean-field term in  a similar way if they have access to the common noise variable. Hence,  if there is no common noise, or if the agents have access to the common noise variable, the infinite population solution can be used by the agents without communication in a fully decentralized way. Similar analysis to what we have in this section can be used to establish the near optimality of these decentralized policies.
\end{remark}

}

\appendix
\section{Proof of Proposition \ref{sym_pol_disc}}\label{sym_pol_disc_proof}

Consider a sequence of measures $\{\mu_n\}_n\subset \P(\mathds{X})$ such that $\mu_n\to \mu$ weakly for some $\mu\in\P(\mathds{X})$. Furthermore, consider the corresponding optimal state-action distribution $\Theta_n(du,dx)$ for $\mu_n$ under the discounted cost criteria for the infinite population problem. Since $\mathds{U}\times\mathds{X}$ is assumed to be compact, there exists a convergent subsequence, say $\Theta_{n'}\to \Theta$. Furthermore, $\Theta$ is an admissible state-action distribution for $\mu$, i.e. $\Theta(\mathds{U},\cdot)=\mu(\cdot)$ since $\mu_n\to \mu$. We claim that the limit $\Theta$ is also an optimal state-action distribution for $\mu$ under the discounted cost criteria. Consider the following Bellman equation:
\begin{align*}
K_\beta^*(\mu_{n'})= k(\mu_{n'},\Theta_{n'}) + \beta \int K_\beta^*(\mu_1)\eta(d\mu_1|\mu_{n'},\Theta_{n'}).
\end{align*}
Under Assumption \ref{lip_reg}, $K_\beta^*$ is continuous, and it can be shown that $k(\mu,\Theta)$ is continuous and $\eta(\cdot|\mu,\Theta)$ is weakly continuous in $\mu,\Theta$. Hence, taking the limit along $n'\to \infty$, 
\begin{align*}
K_\beta^*(\mu) = k(\mu,\Theta) + \beta \int K_\beta^*(\mu_1)\eta(d\mu_1|\mu,\Theta)
\end{align*}
which proves the claim that $\Theta$ is an optimal action for $\mu$. In other words, any for some $\mu_n\to \mu$ and $\Theta_n$ such that $\Theta_n$ is optimal for $\mu_n$, limit any convergent subsequence of $\Theta_n$ will be optimal for $\mu$. 

We now define the following operator for any function $h:\P_N(\mathds{\mathds{X}}) \to \mathds{R}$ for the finite population problem such that
\begin{align*}
(T^N h)(\mu^N) = k(\mu^N,g_\beta(\mu^N)) + \beta \int h(\mu_1^N)\eta(d\mu_1^N|\mu^N,g_\beta(\mu^N))
\end{align*}
where $g_\beta(\mu^N) = \gamma(du|x,\mu^N)\mu^N(dx)$, i.e. $g_\beta$ is the resulting team policy when the agents use symmetric policies from the infinite population problem. Note that $T^N$ is a contraction and 
\begin{align*}
\lim_{k\to \infty} (T^N_k h) (\mu^N)= K_\beta^N(\mu^N,g_\beta)
\end{align*}
where $T_k^N$ denotes $T^N$ applied to the function $h$, $k$ times. The convergence above is uniform over $N$ since contraction modulus $\beta$ is independent of $N$ and since the stage-wise cost $c(c,u,\mu)$ is assumed to be uniformly bounded.

We claim that 
$
\lim_{N\to\infty} T^N_k (K_\beta^{N,*})(\mu^N) = K_\beta^*(\mu)
$
for all $k<\infty$ for all $\mu^N\to \mu$. We prove the claim by induction. For $k=1$:
\begin{align*}
T^N K_\beta^{N,*}(\mu^N) = k(\mu^N,g_\beta(\mu^N)) + \beta \int K_\beta^{N,*}(\mu^N_1)\eta(d\mu_1^N|\mu^N,g_\beta(\mu^N)).
\end{align*}
We assume that there exists some $\epsilon>0$ and a subsequence $N'$ such that 
\begin{align*}
\left|T^{N'} K_\beta^{N',*}(\mu^{N'})  - K_\beta^*(\mu)\right|>\epsilon.
\end{align*}
For this subsequence, there exists a further subsequence, say $N_m$, such that $g_\beta(\mu^{N_m})$ converges to some $\Theta$ that is optimal for $\mu$ by the initial argument we had at  the start of the proof. Hence, taking the limit of both sides along $N^{m}\to \infty$ and using Lemma \ref{key_lemma} with the fact that $K_\beta^{N,*}(\mu^N)\to K_\beta^*(\mu)$, we can have that
\begin{align*}
\lim_{N^{m}} T^{N^{m}} K_\beta^{{N^m},*}(\mu^{N^{m}}) = k(\mu,\Theta) + \int K_\beta^*(\mu_1)\eta(d\mu_1|\mu,\Theta)=K_\beta^*(\mu)  
\end{align*}
where the last equality follows since $\Theta$ is an optimal action for $\mu$. Finally, since the value function $K_\beta^*$ is unique, we reach a contradiction, which implies that $
\lim_{N\to\infty}( T^N K_\beta^{N,*})(\mu^N) = K_\beta^*(\mu)$.

For $k+1$:
\begin{align*}
T^N_k K_\beta^{N,*}(\mu^N) = k(\mu^N,g_\beta(\mu^N)) + \beta \int T^N_k(K_\beta^{N,*})(\mu^N_1)\eta(d\mu_1^N|\mu^N,g_\beta(\mu^N)).
\end{align*}
Using the same contradiction argument with the induction step that $ \lim_{N\to \infty}T^N_k(K_\beta^{N,*})(\mu^N)=K_\beta^*(\mu)$ we can conclude that 
\begin{align*}
\lim_{N\to\infty}( T_k^N K_\beta^{N,*})(\mu^N) = K_\beta^*(\mu), \quad \forall k.
\end{align*}

Finally, we write
\begin{align*}
\left|K_\beta^N(\mu^N,g_\beta)- K_\beta^*(\mu)\right| \leq \left|K_\beta^N(\mu^N,g_\beta)- ( T_k^N K_\beta^{N,*})(\mu^N) \right| +  \left|(T_k^N K_\beta^{N,*})(\mu^N) -K_\beta^*(\mu)\right|
\end{align*}
where the first term can be made arbitrarily small by choosing $k$ large enough for all $N$, and the the chosen $k$ the second term is shown to converge to $0$ as $N\to \infty$. This proves the first part of the result. The second part follows since we have $K_\beta^{N,*}(\mu^N)\to K_\beta^*(\mu)$ for any $\mu^N\to \mu$ weakly.

\section{Proof of Lemma \ref{key_lemma}}\label{key_lemma_proof}
For the given $\mu^N$ and $\Theta^N$, one can find state and action vectors ${\bf x} = [x^1,x^2,\dots,x^N]$ and  ${\bf u} = [u^1,u^2,\dots,u^N]$ such that $\mu_{\bf x}=\mu^N$ and $\mu_{(\bf x,u)}=\Theta^N$ where $\mu_{\bf x}$ denotes the empirical distribution of the vector ${\bf x}$.  Note that for any permutation $\sigma({\bf x,u})$, we will have that $\mu_{\sigma({\bf x})} = \mu_{\bf x}=\mu^N$ and $\mu_{\sigma({\bf x,u})} = \mu_{\bf(x,u)}=\Theta^N$

We also consider some ${\bf \hat{x}, \hat{u}} = [(x^1,u^1),\dots, (x^N,u^N)]$ such that  $(x^i,u^i)\sim \Theta(dx,du)$ for all $i= 1,\dots,N$.  We then have
$
W_1(\mu_{\bf \hat{x},\hat{u}}, \mu_{\bf x,u})= \min_\sigma\frac{1}{N}\sum_{i=1}^N \left| (\hat{x}^i,\hat{u}^i) - \sigma(x^i,u^i)\right|.
$
Note that the minimum above is achievable that is a particular permutation of the state-action vector ${\bf (x,u)}$ achieves the minimum. In what follows we consider this particular permutation.

We start by writing
\begin{align*}
&\left|\int g^N(\mu_1^N)\eta(d\mu_1^N|\mu^N,\Theta^N) - \int g(\mu_1)\eta(d\mu_1|\mu,\Theta)\right|\\
&\leq \left|\int g^N(\mu_1^N)\eta^N(d\mu_1^N|\mu_{\bf x},\mu_{\bf x,u}) - \int g^N(\mu_1^N)\eta(d\mu^N_1|\mu_{\bf \hat{x}},\mu_{\bf \hat{x},\hat{u}})\right|\\
&\quad + \left|  \int g^N(\mu_1^N)\eta(d\mu_1^N|\mu_{\bf \hat{x}},\mu_{\bf \hat{x},\hat{u}})- \int g(\mu_1)\eta(d\mu_1|\mu,\Theta)\right|.
\end{align*}
For the first term, we have that 
\begin{align*}
& \left|\int g^N(\mu_1^N)\eta(d\mu_1^N|\mu_{\bf x},\mu_{\bf x,u}) - \int g^N(\mu_1^N)\eta(d\mu_1|\mu_{\bf \hat{x}},\mu_{\bf \hat{x},\hat{u}})\right|\\
&= \left|\int g^N(\mu^N_{f({\bf x ,u, w})}) P(d{\bf w}) - \int  g^N(\mu^N_{f({\bf \hat{x} ,\hat{u}, w})}) P(d{\bf w}) \right|\\
&\leq K \int  W_1(\mu^N_{f({\bf x ,u, w})}, \mu^N_{f({\bf \hat{x} ,\hat{u}, w})}) P({\bf dw})\\
&\leq K \int \frac{1}{N} \sum_{i=1}^N \left|f(x^i,u^i,\mu_{\bf x},w^i,w^0) - f(\hat{x}^i,\hat{u}^i,\mu_{\bf \hat{x}},w^i,w^0)\right| P(d{\bf w}) \\
&\leq K \frac{1}{N}\sum_{i=1}^N \left( |x^i- \hat{x}^i| + |u^i-\hat{u}^i| + W_1(\mu_{\bf x},\mu_{\bf \hat{x}})  \right)= K \left(W_1(\mu_{ \bf x,u} ,\mu_{\bf \hat{x}, \hat{u}}) + W_1(\mu_{\bf x},\mu_{\bf \hat{x}})\right)\to 0.
\end{align*}
where $\mu^N_{f({\bf x ,u, w})}$ denotes the empirical distribution of the upcoming local state variables of the agents, and $P(d{\bf w})$ is simply the probability measure for the local noise variables of the agents $\{w_1,\dots,w^N\}$ and the common noise $w^0$. The last inequality follows since we use a particular permutation of $f({\bf x ,u, w})= \{f(x^i,u^i,\mu_{\bf x},w^i,w^0) \}_{i=1}^N$. The last equality follows since the particular permutation fixed at the start is the one that achieves $W_1(\mu_{ \bf x,u} ,\mu_{\bf \hat{x}, \hat{u}}) $. Finally, the convergence to 0 holds since $\mu_{\bf x,u}=\Theta^N\to \Theta$ by assumption, and $\mu_{\bf \hat{x},\hat{u}}\to \Theta$ since they form an empirical measure for $\Theta$. 

We now focus on the second term:
$
 \left|  \int g^N(\mu_1^N)\eta(d\mu_1|\mu_{\bf \hat{x}},\mu_{\bf \hat{x},\hat{u}})- \int g(\mu_1)\eta(d\mu_1|\mu,\Theta)\right|.$ Note that the local noise of the agents, i.e. $w^i$, are i.i.d.. We consider an extended probability space for the infinite sequence of the local noise variables $\{w^1,w^2,\dots,w^N,\dots\}$, in which every $\omega\in \Omega$ gives rise to a possibly  different sequence.  Then we write the above term as
\begin{align*}
\int  g^N(\mu^N_{f({\bf \hat{x} ,\hat{u}, w})})\prod_{i=1}^NP(dw^i)P(dw^0) = \int  g^N(\mu^N_{f({\bf \hat{x} ,\hat{u}, w(\omega)})}) P(d\omega)P(dw^0) 
\end{align*}
where ${\bf w}(\omega)$ is simply the first $N$ terms of the local noise variables and the common noise variable.  Observe that for a given common noise $w^0$ and a sequence of local noise variables (or a given $\omega$), $\mu_1^N=\mu^N_{f({\bf \hat{x} ,\hat{u}, w})}$ is the empirical distribution of  $\{f(\hat{x}^i,\hat{u}^i,\mu_{\bf \hat{x}},w^i,w^0)\}_i$ where $(\hat{x}^i,\hat{u}^i)\sim \Theta$, and the local noise $w^i$ are i.i.d.. Hence, for a given common noise $w^0$,  and for any continuous function $h$ with unit Lipschitz constant, we have that 
\begin{align*}
&\int h(x_1)\mu_1^N(dx_1) = \frac{1}{N}\sum_{i=1}^N h(f(\hat{x}^i,\hat{u}^i,\mu_{\bf \hat{x}},w^i_{(\omega)},w^0))\pm h(f(\hat{x}^i,\hat{u}^i,\mu,w^i_{(\omega)},w^0))\\
&\leq K_f W_1(\mu_{\bf x},\mu) +  \frac{1}{N}\sum_{i=1}^N h(f(\hat{x}^i,\hat{u}^i,\mu,w^i_{(\omega)},w^0)) \\
&\to \int h(f(x,u,\mu,w,w^0))\Theta(dx,du)P(dw) = \int h(x_1)\mathcal{T}^{w^0}(dx_1|x,u,\mu)\Theta(dx,du)
\end{align*}
for almost every $\omega\in \Omega$. 
For any given common noise realization $\mu_1^N(dx_1)$ converges to $\int \mathcal{T}^{w^0}(dx_1|x,u,\mu)\Theta(dx,du)$ almost surely. Hence by definition of the kernel $\eta$ of the infnite population problem,  and by the assumption that $g^N(\mu^N)\to g(\mu)$ for all $\mu^N\to\mu$ we can conclude 
\begin{align*}
 \int  g^N(\mu^N_{f({\bf \hat{x} ,\hat{u}, w(\omega)})}) P(d\omega)P(dw^0) \to \int g(\mu_1(\mu,\Theta,w^0))P(dw^0) =\int g(\mu_1)\eta(d\mu_1|\mu,\Theta)
\end{align*}
which conclude the proof.

\bibliographystyle{abbrv}

{\footnotesize
\bibliography{mfc_bibliography}}

\end{document}